\documentclass[10pt]{amsart}

\usepackage{amssymb, amsmath, amsthm, amsfonts}
\usepackage{mathrsfs,comment}
\usepackage{graphicx}
\usepackage{placeins}
\usepackage{rotating}
\usepackage{tikz}
\usepackage{float}
\usepackage{subfigure}
\usepackage{hvfloat}
\usepackage{caption}
\usepackage{pdflscape}
\usepackage{hyperref}  
\usepackage{url}
\usepackage[all,arc,2cell]{xy}
\UseAllTwocells
\usepackage{enumerate}
\usepackage{chngcntr}
 \usepackage{lineno}
 \usepackage{blindtext}
\usepackage{verbatim}
\usepackage{soul}
\usepackage[normalem]{ulem}

\usepackage{lscape}
\usepackage{geometry}

\usepackage[makeroom]{cancel}

\hypersetup{%
  bookmarksnumbered=true,%
  bookmarks=true,%
  colorlinks=true,%
  linkcolor=blue,%
  citecolor=blue,%
  filecolor=blue,%
  menucolor=blue,%
  pagecolor=blue,%
  urlcolor=blue,%
  pdfnewwindow=true,%
  pdfstartview=FitBH}

\def\@url#1{{\tt\def~{\lower3.5pt\hbox{\char'176}}\def\_{\char'137}#1}}

\let\fullref\autoref
\def\makeautorefname#1#2{\expandafter\def\csname#1autorefname\endcsname{#2}}
\makeautorefname{equation}{Equation}%
\makeautorefname{footnote}{footnote}%
\makeautorefname{item}{item}%
\makeautorefname{figure}{Figure}%
\makeautorefname{table}{Table}%
\makeautorefname{part}{Part}%
\makeautorefname{appendix}{Appendix}%
\makeautorefname{chapter}{Chapter}%
\makeautorefname{section}{Section}%
\makeautorefname{subsection}{Section}%
\makeautorefname{subsubsection}{Section}%
\makeautorefname{paragraph}{Paragraph}%
\makeautorefname{subparagraph}{Paragraph}%
\makeautorefname{theorem}{Theorem}%
\makeautorefname{theo}{Theorem}%
\makeautorefname{thm}{Theorem}%
\makeautorefname{addendum}{Addendum}%
\makeautorefname{addend}{Addendum}%
\makeautorefname{add}{Addendum}%
\makeautorefname{maintheorem}{Main theorem}%
\makeautorefname{mainthm}{Main theorem}%
\makeautorefname{corollary}{Corollary}%
\makeautorefname{claim}{Claim}%
\makeautorefname{corol}{Corollary}%
\makeautorefname{coro}{Corollary}%
\makeautorefname{cor}{Corollary}%
\makeautorefname{lemma}{Lemma}%
\makeautorefname{lemm}{Lemma}%
\makeautorefname{lem}{Lemma}%
\makeautorefname{sublemma}{Sublemma}%
\makeautorefname{sublem}{Sublemma}%
\makeautorefname{subl}{Sublemma}%
\makeautorefname{proposition}{Proposition}%
\makeautorefname{proposit}{Proposition}%
\makeautorefname{propos}{Proposition}%
\makeautorefname{propo}{Proposition}%
\makeautorefname{prop}{Proposition}%
\makeautorefname{property}{Property}
\makeautorefname{proper}{Property}
\makeautorefname{scholium}{Scholium}%
\makeautorefname{step}{Step}%
\makeautorefname{conjecture}{Conjecture}%
\makeautorefname{conject}{Conjecture}%
\makeautorefname{conj}{Conjecture}%
\makeautorefname{question}{Question}
\makeautorefname{questn}{Question}
\makeautorefname{quest}{Question}
\makeautorefname{ques}{Question}
\makeautorefname{qn}{Question}
\makeautorefname{definition}{Definition}%
\makeautorefname{defin}{Definition}%
\makeautorefname{defi}{Definition}%
\makeautorefname{def}{Definition}%
\makeautorefname{defn}{Definition}%
\makeautorefname{dfn}{Definition}%
\makeautorefname{notation}{Notation}
\makeautorefname{nota}{Notation}
\makeautorefname{notn}{Notation}
\makeautorefname{remark}{Remark}%
\makeautorefname{rema}{Remark}%
\makeautorefname{rem}{Remark}%
\makeautorefname{rmk}{Remark}%
\makeautorefname{rk}{Remark}%
\makeautorefname{remarks}{Remarks}%
\makeautorefname{rems}{Remarks}%
\makeautorefname{rmks}{Remarks}%
\makeautorefname{rks}{Remarks}%
\makeautorefname{example}{Example}%
\makeautorefname{examp}{Example}%
\makeautorefname{exmp}{Example}%
\makeautorefname{exmps}{Examples}%
\makeautorefname{exam}{Example}%
\makeautorefname{exa}{Example}%
\makeautorefname{algorithm}{Algorith}%
\makeautorefname{algo}{Algorith}%
\makeautorefname{alg}{Algorith}%
\makeautorefname{axiom}{Axiom}%
\makeautorefname{axi}{Axiom}%
\makeautorefname{ax}{Axiom}%
\makeautorefname{case}{Case}%
\makeautorefname{claim}{Claim}%
\makeautorefname{clm}{Claim}%
\makeautorefname{assumption}{Assumption}%
\makeautorefname{assumpt}{Assumption}%
\makeautorefname{conclusion}{Conclusion}%
\makeautorefname{concl}{Conclusion}%
\makeautorefname{conc}{Conclusion}%
\makeautorefname{condition}{Condition}%
\makeautorefname{condit}{Condition}%
\makeautorefname{cond}{Condition}%
\makeautorefname{construction}{Construction}%
\makeautorefname{construct}{Construction}%
\makeautorefname{const}{Construction}%
\makeautorefname{cons}{Construction}%
\makeautorefname{criterion}{Criterion}%
\makeautorefname{criter}{Criterion}%
\makeautorefname{crit}{Criterion}%
\makeautorefname{exercise}{Exercise}%
\makeautorefname{exer}{Exercise}%
\makeautorefname{exe}{Exercise}%
\makeautorefname{problem}{Problem}%
\makeautorefname{problm}{Problem}%
\makeautorefname{probm}{Problem}%
\makeautorefname{prob}{Problem}%
\makeautorefname{solution}{Solution}%
\makeautorefname{soln}{Solution}%
\makeautorefname{sol}{Solution}%
\makeautorefname{summary}{Summary}%
\makeautorefname{summ}{Summary}%
\makeautorefname{sum}{Summary}%
\makeautorefname{operation}{Operation}%
\makeautorefname{oper}{Operation}%
\makeautorefname{observation}{Observation}%
\makeautorefname{observn}{Observation}%
\makeautorefname{obser}{Observation}%
\makeautorefname{obs}{Observation}%
\makeautorefname{ob}{Observation}%
\makeautorefname{convention}{Convention}%
\makeautorefname{convent}{Convention}%
\makeautorefname{conv}{Convention}%
\makeautorefname{cvn}{Convention}%
\makeautorefname{warning}{Warning}%
\makeautorefname{warn}{Warning}%
\makeautorefname{note}{Note}%
\makeautorefname{fact}{Fact}%

  \makeatletter
                   \let\c@lemma\c@theorem
                  \makeatother

\newtheorem{thm}{Theorem}[subsection]
\newtheorem{cor}{Corollary}[subsection]
\newtheorem{prop}{Proposition}[subsection]
\newtheorem{lem}{Lemma}[subsection]
\theoremstyle{definition}
\newtheorem{defn}{Definition}[subsection]
\newtheorem{exmp}{Example}[subsection]

\newtheorem{rem}{Remark}[subsection]
\newtheorem{warn}{Warning}[subsection]

\newtheorem{convention}{Convention}[subsection]

\newtheorem{notation}{Notation}[subsection]

\makeatletter
\let\c@lem=\c@thm
\let\c@cor=\c@thm
\let\c@prop=\c@thm
\let\c@lem=\c@thm
\let\c@defn=\c@thm
\let\c@exmps=\c@thm
\let\c@rem=\c@thm
\let\c@warn=\c@thm
\let\c@claim=\c@thm
\let\c@quest=\c@thm
\let\c@notation=\c@thm
\makeatother

\numberwithin{equation}{subsection}

\newcommand{\A}{\mathbb{A}}

\newcommand{\Z}{\mathbb{Z}}

\newcommand{\spi}{\underline{\pi}}

\DeclareSymbolFontAlphabet{\scr}{rsfs}

\newcommand{\smsh}{\wedge}
\newcommand{\ra}{\rightarrow}
\newcommand{\xra}{\xrightarrow}

\newcommand{\SH}[1]{\mathcal{SH}({#1})}

\newcommand{\Ho}{\operatorname{Ho}}

\def\quickop#1{\expandafter\newcommand\csname #1\endcsname{\operatorname{#1}}}
\quickop{Hom} \quickop{End} \quickop{Aut} \quickop{Tel} \quickop{Mic} 
\quickop{Ext} \quickop{Tor} \quickop{Cotor} \quickop{Id} \quickop{Coker} \quickop{Ker}
\quickop{Lim} \quickop{Colim} \quickop{Holim} \quickop{Hocolim}
\quickop{id} \quickop{tel} \quickop{mic} \quickop{coker} 
\quickop{colim} \quickop{holim} \quickop{hocolim} \quickop{im}
\DeclareMathOperator{\Gal}{Gal}

\DeclareMathOperator{\Tot}{Tot}

\newcommand{\KQ}{\mathsf{KO}}

\newcommand{\KGL}{\mathsf{KGL}}
\newcommand{\BGL}{\mathsf{BGL}}
\newcommand{\MGL}{\mathsf{MGL}}
\newcommand{\mmS}{\mathbb S}
\newcommand{\KT}{\mathsf{KT}}

\newcommand{\HR}{\mathsf{HR}}
\newcommand{\HT}{\mathsf{HT}}
\newcommand{\HA}{\mathsf{HA}}
\newcommand{\HB}{\mathsf{HB}}

\newcommand{\EGgeom}{\mathbb{E}G}

\newcommand{\EGg}[1]{\mathbb{E}{#1}}

\DeclareMathOperator{\Spec}{Spec}

\newcommand{\cat}{\mathsf}
\newcommand{\vp}{\varphi}
\newcommand{\ve}{\varepsilon}
\newcommand{\triv}[1]{{\operatorname{Triv}_{#1}}}
\newcommand{\lot}[1]{\otimes ^{\mathbb L}_{#1}}
\newcommand{\rhom}[1]{\mathbb R\negthinspace \operatorname{Hom}_{#1}}
\newcommand{\ho}{\operatorname{Ho}}
\newcommand{\htriv}[1]{\operatorname{Triv}_{#1}^{h}}
\newcommand{\hfp}[1]{(-)^{h#1}}
\newcommand{\Hfp}[2]{#2^{h#1}}
\newcommand{\horb}[1]{(-)_{h#1}}
\newcommand{\Horb}[2]{#2_{h#1}}
\newcommand{\Gmot}{G\text{-}\spre (\Sm{S})}
\newcommand{\pGmot}{G\text{-}\spre_* (\Sm{S})}

\newcommand{\motG}{\spre (\Sm{S})\text{-}G}
\newcommand{\pmotG}{\spre_{*} (\Sm{S})\text{-}G}

\newcommand{\WE}{\mathcal{W}}
\newcommand{\Fib}{\mathcal{F}}
\newcommand{\Cof}{\mathcal{C}}

\DeclareMathOperator{\LLP}{LLP}
\DeclareMathOperator{\RLP}{RLP} 
\newcommand{\sK}{\mathsf K}
\newcommand{\sM}{\mathsf M}
\newcommand{\sN}{\mathsf N}
\newcommand{\sP}{\mathsf P}
\newcommand{\sfC}{\mathsf C}

\newcommand{\spre}{\mathsf{sPre}}

\newcommand{\trivial}{\mathsf{Triv}}
\newcommand{\spe}{\mathsf{Sp}^\Sigma}
\newcommand{\ti}{\times}

\newcommand{\bA}{\mathbb{A}}
\newcommand{\lra}{\longrightarrow}
\newcommand{\alg}{\mathsf{Alg}}
\newcommand{\calg}{\mathsf{CAlg}}
\newcommand{\CC}{\mathsf{C}}

\newcommand\adjunction[4]{\xymatrix{#1\ar @<1.18ex>[rr]^-{#3}&\perp&#2\ar @<1.18ex>[ll]^-{#4}}}

\definecolor{darkspringgreen}{rgb}{0.09, 0.45, 0.27}
\definecolor{darkterracotta}{rgb}{0.8, 0.31, 0.36}
	\definecolor{darkcoral}{rgb}{0.8, 0.36, 0.27}
	\definecolor{indiagreen}{rgb}{0.07, 0.53, 0.03}
	\definecolor{mountainmeadow}{rgb}{0.19, 0.73, 0.56}
	\definecolor{mountbattenpink}{rgb}{0.6, 0.48, 0.55}
	\definecolor{palatinatepurple}{rgb}{0.41, 0.16, 0.38}
	\definecolor{cinnamon}{rgb}{0.82, 0.41, 0.12}
	\definecolor{chocolate}{rgb}{0.82, 0.41, 0.12}
	
\newcommand{\sSet}{\mathsf{sSet}}
\newcommand{\psSet}{\mathsf{sSet}_*}

\newcommand{\Pre}[1]{\mathsf{Pre}(#1)} 
\newcommand{\sPre}[1]{\mathsf{sPre}(#1)} 
\newcommand{\psPre}[1]{\mathsf{sPre}_*(#1)} 

\newcommand{\Sm}[1]{\mathsf{Sm}_{/#1}} 

\newcommand{\aspreS}{\mathsf{Alg}(S)}

\newcommand{\caspreS}{\mathsf{CAlg}(S)}

\newcommand{\aGspreS}{G\text{-}\mathsf{Alg}(S)}

\newcommand{\caGspreS}{G\text{-}\mathsf{CAlg}(S)}

\newcommand{\aspreSG}{\mathsf{Alg}(S)\text{-}G}

\newcommand{\caspreSG}{\mathsf{CAlg}(S)\text{-}G}

\newcommand{\spc}{\mathsf{Spc}(S)}
\newcommand{\pspc}{{\mathsf{Spc}_*}(S)}
\newcommand{\Gspc}{G\text{-}\mathsf{Spc}(S)}
\newcommand{\Gpspc}{G\text{-}{\mathsf{Spc}_*}(S)}
\newcommand{\Gspcl}{G\text{-}\mathsf{Spc}(S)_{\mathrm {left}}}
\newcommand{\pGspcl}{{G\text{-}\mathsf{Spc}_*}(S)_{\mathrm {left}}}
\newcommand{\Gspcr}{G\text{-}\mathsf{Spc}(S)_{\mathrm {right}}}

\newcommand{\spcG}{\mathsf{Spc}(S)\text{-}G}

\newcommand{\spclG}{\mathsf{Spc}(S)\text{-}G_{\mathrm {left}}}
\newcommand{\spcrG}{\mathsf{Spc}(S)\text{-}G_{\mathrm {right}}}

\newcommand{\aspc}{\mathsf{AlgSpc}(S)}

\newcommand{\aGspcl}{G\text{-}\mathsf{AlgSpc}(S)_{\mathrm {left}}}

\newcommand{\aspcGr}{\mathsf{AlgSpc}(S)\text{-}G_{\mathrm {right}}}

\newcommand{\caspc}{\mathsf{CAlgSpc}(S)}

\newcommand{\caGspcl}{G\text{-}\mathsf{CAlgSpc}(S)_{\mathrm {left}}}

\newcommand{\caspcGr}{\mathsf{CAlgSpc}(S)\text{-}G_{\mathrm {right}}}

\newcommand{\spt}{\mathsf{Sp}_{X}(S)}

\newcommand{\Gspt}{G\text{-}\mathsf{Sp}_{X}(S)}
\newcommand{\Gsptr}{G\text{-}\mathsf{Sp}_{X}(S)_{\mathrm {right}}}
\newcommand{\Gsptl}{G\text{-}\mathsf{Sp}_{X}(S)_{\mathrm {left}}}

\newcommand{\sptrG}{\mathsf{Sp}_{X}(S)\text{-}G_{\mathrm {right}}}

\newcommand{\aspt}{\mathsf{AlgSp}_{X}(S)}
\newcommand{\aGspt}{G\text{-}\mathsf{AlgSp}_{X}(S)}
\newcommand{\aGsptl}{G\text{-}\mathsf{AlgSp}_{X}(S)_{\mathrm {left}}}
\newcommand{\asptG}{\mathsf{AlgSp}_{X}(S)\text{-}G}
\newcommand{\asptrG}{\mathsf{AlgSp}_{X}(S)\text{-}G_{\mathrm {right}}}

\newcommand{\caGspt}{G\text{-}\mathsf{CAlgSp}_{X}(S)}
\newcommand{\casptG}{\mathsf{CAlgSp}_{X}(S)\text{-}G}

\newcommand{\Pspt}{{\mathsf{Sp}_{X}(S)}_{\mathrm{pos}}}
\newcommand{\caspt}{\mathsf{CAlgSp}_{X}(S)}
\newcommand{\Pcaspt}{\mathsf{CAlgSp}_{X}(S)_{\mathrm{pos}}}

\newcommand{\GPsptr}{G\text{-}\mathsf{Sp}_{X}(S)_{\mathrm {pos,{right}}}}
\newcommand{\GPsptl}{G\text{-}\mathsf{Sp}_{X}(S)_{\mathrm {pos,{left}}}}
\newcommand{\PsptrG}{\mathsf{Sp}_{X}(S)\text{-}G_{\mathrm {pos,{right}}}}

\newcommand{\caGPsptl}{{G\text{-}\mathsf{CAlgSp}_{X}(S)}_{\mathrm{pos,{left}}}}

\newcommand{\caPsptrG}{{\mathsf{CAlgSp}_{X}(S)\text{-}G}_{\mathrm{pos,{right}}}}
\newcommand{\caPsptG}{\mathsf{CAlgPSp}_{X}(S)\text{-}G}

\title{Motivic homotopical Galois extensions}
\author[Beaudry]{Agn\`es Beaudry}
\address[Beaudry]{Department of Mathematics, University of Colorado at Boulder}
\author[Hess]{Kathryn Hess}
\address[Hess]{SV UPHESS BMI, \'Ecole Polytechnique F\'ed\'erale de Lausanne}
\author[Kedziorek]{Magdalena Kedziorek}
\address[Kedziorek]{SV UPHESS BMI, \'Ecole Polytechnique F\'ed\'erale de Lausanne}
\author[Merling]{Mona Merling}
\address[Merling]{Department of Mathematics, Johns Hopkins University}
\author[Stojanoska]{Vesna Stojanoska}
\address[Stojanoska]{Department of Mathematics, University of Illinois at Urbana-Champaign}

\begin{document}

\begin{abstract}
We establish a formal framework for Rognes's homotopical Galois theory and adapt it to the context of motivic spaces and spectra. We discuss examples of Galois extensions between Eilenberg-MacLane motivic spectra and between the Hermitian and algebraic $K$-theory spectra.
\end{abstract}

\maketitle

\tableofcontents

\section{Introduction}
Voevodsky's solution to the Bloch-Kato conjecture (\cite{voe1}, \cite{voe2}) established the importance of motivic homotopy theory as a tool for algebraic geometry. Many classical results of homotopy theory have since been translated to the motivic world. In this spirit, this paper adapts homotopical Galois theory to the context of motivic spaces and spectra.

In motivic homotopy theory, one studies algebraic varieties from a homotopy-theoretic perspective. For each base scheme $S$, there is a stable motivic homotopy category $\SH{S}$, analogous to the classical stable homotopy category, whose objects represent homotopy-invariant cohomology theories for algebraic varieties over $S$. Here, homotopies are defined using the affine line as a substitute for the unit interval. From an abstract point of view, $\SH{S}$ is the homotopy category of a stable, symmetric monoidal model category, making it amenable to techniques from model category theory. 

The theory of Galois extensions of rings, introduced by Chase, Harrison, and Rosenberg \cite{c-h-r} in the early 1960's and further elaborated by Knus and Ojanguran \cite{knus-ojanguran} ten years later, generalizes Galois theory of fields. Inspired by Rognes's homotopical generalization of the theory of Galois extensions to ring spectra \cite{rognes}, we develop here an analogous theory for motivic ring spaces and spectra, establish a number of important properties of motivic Galois extensions, and provide concrete examples of motivic Galois extensions.

\subsection*{Description of the results}

We first establish an abstract framework for derived Galois theory, motivated by \cite{rognes}.  Let $(\cat M, \otimes, I)$ be a symmetric monoidal model category. There are well known conditions under which there are induced model category structures on the category $\cat {Alg}$ of algebras in $\cat M$, as well as on the categories of left and right modules over any algebra in $\cat M$, so that every morphism of algebras $A\to B$ induces an extension/restriction-of-coefficients adjunction
$$\adjunction{\Ho \cat {Mod}_{A}}{\Ho \cat {Mod}_{B}} {-\lot A B}{\rhom A(B,-)}.$$  Moreover, if the model category structure on $\cat M$ is nice enough, for all bialgebras $H$ in $\cat M$, there is an induced model category structure on the category ${}_{H}\cat {Alg}$ of left $H$--algebras such that
the trivial $H$--module functor $\operatorname{Triv}_{H}\colon \cat {Alg} \to {}_{H}\cat {Alg}$ and its right adjoint, the $H$-fixed points functor   $(-)^{H}\colon {}_{H}\cat {Alg} \to \cat {Alg}$, form a Quillen pair, and thus descend to an adjunction 
$$\adjunction{\Ho \cat {Alg}}{\Ho {}_{H}\cat {Alg}}{\htriv H}{\hfp H}.$$ 
Similarly, there is an induced extension/restriction-of-coefficients adjunction
$$\adjunction{\Ho \cat M}{\Ho \cat M} {-\lot{} H}{\rhom{}(H,-)}.$$

Given the existence of the adjunctions above, we define \emph{homotopical Galois data} (\fullref{defn:galois-data}) in the monoidal model category $(\cat M, \otimes , I)$ to consist of
\begin{itemize}
\item a dualizable Hopf algebra $H$ in $\cat M$,
\item an algebra $A$ and an $H$--algebra $B$ in $\cat M$, and 
\item a morphism $\vp\colon \operatorname{Triv}_{H}A \to B$ in $\ho ({}_{H}\cat {Alg})$.
\end{itemize} 
We denote this data by $\vp:  \triv{H}{A}\to B^{\circlearrowleft H}$. 

For Galois data $\vp:  \triv{H}{A}\to B^{\circlearrowleft H}$, let   $\beta_{\vp}\colon B\lot A B \to \rhom{}(H,B)$ be the transpose in the homotopy category of the map $B\lot A B\lot{} H \to B$ given by the $H$--action and the multiplication on $B$.
If both $\vp^{hH}\colon A \to B^{hH}$ and $\beta_{\vp}\colon B\lot AB \to \rhom{} (H,B)$ are isomorphisms, then $\vp\colon \operatorname{Triv}_{H}A \to B$ is a \emph{homotopical $H$--Galois extension} (\fullref{defn:galois-ext}).

It follows easily from the definition that if $H$ is a dualizable Hopf algebra (\fullref{defn:dualizable}) that is cofibrant in $\cat M$, then its counit $\ve\colon H\to I$ induces a homotopical $H$-Galois extension $\triv HA \to \operatorname{Hom}(H, A)^{\circlearrowleft H}$ for all algebras $A$ (\fullref{exmp:trivial-ext}). Under the conditions of \fullref{conv:modelcat2},  we show  that if $ \triv{H}{A}\to B^{\circlearrowleft H}$ is a homotopical $H$-Galois extension, then $B$ is dualizable as an $A$-module, and characterize faithful homotopical Galois extensions in terms of dualizability (\fullref{prop:galois-dualizable} and \fullref{prop:characterize}). We can then prove invariance of homotopical Galois extensions of commutative algebras under extension of coefficients (\fullref{sec:cobasechange}) and establish the forward part of a Galois correspondence, obtaining a homotopical $K$--Galois subextension of  any faithful homotopical $H$--Galois extension of commutative algebras for every ``allowable'' Hopf algebra map $K\to H$ (\fullref{thm:galois-corresp}).

Fixing a smooth scheme $S$ and a pointed motivic space $X$, we focus in \fullref{sec:motmodstructures} on the case where $\cat M$ is the category of (pointed) motivic spaces over $S$, denoted $\spc$ (respectively, $\pspc$), or the category of motivic $X$-spectra equipped with its usual stable model structure, denoted $\spt$.  A motivic $X$-spectrum is a sequence $(Y_{0},Y_{1},..., Y_{n},...)$ of pointed motivic spaces, where $Y_{n}$ is equipped with a $\Sigma_{n}$-action for every $n$, together with structure maps $Y_{n}\wedge X \to Y_{n+1}$ that are appropriately $\Sigma_{n+1}$-equivariant. For any finite group $G$, which can naturally be seen  as a bialgebra in any of these categories,  we equip with model structures the category  of (pointed) motivic $G$-spaces, $\Gspc$ (respectively, $\Gpspc$), i.e.,  of $G$-objects in $\spc$ (respectively, in $\pspc$), and the category $\Gspt$ of motivic $\triv GX$-spectra of motivic $G$-spaces.  In both the nonequivariant and the equivariant cases, we establish the existence of induced model structures on categories of modules over algebras, of algebras, and of commutative algebras satisfying the conditions of \fullref{conv:modelcat2}.    We conclude that (pointed) motivic spaces and spectra fit into the formal framework for homotopical Galois theory of  \fullref{sec:gal-formal} (cf. \fullref{sec:motivic-framework}).

Finally, in \fullref{sec:examples} we provide concrete examples of motivic Galois extensions analogous to known Galois extensions of classical spectra. We first consider Galois extensions of motivic Eilenberg-MacLane spectra. For any abelian group $A$, let $\HA$ denote the \emph{motivic} Eilenberg-MacLane spectrum, i.e., the representing object for motivic cohomology with coefficients in $A$.  For any homomorphism of commutative rings $R\to T$ and any subgroup $G$ of $\operatorname{Aut}_{R}(T)$, we show that $\HR\to \HT$ is a homotopical $G$-Galois extension if and only if $R\to T$ is $G$-Galois in the classical sense (\fullref{thm:galEM}). Our second example is analogous to the classical extension $KO \to KU$. Based on major results of Hu, Kriz and Ormsby in \cite{hko} and Berrick, Karoubi, Schlichting and {{\O}}stv\ae{}r \cite{bako} on Thomason's  homotopy limit problem \cite{thomason}, 
we state conditions under which the extension from $\KQ$, the motivic spectrum representing Hermitian $K$-theory, to $\KGL$, the motivic spectrum representing algebraic $K$-theory, is a homotopical $C_{2}$-Galois extension and prove that, in this case, it is faithful on $\eta$--complete modules (\fullref{thm:simKQ}).

\subsection*{Related and future work}
Equivariant motivic homotopy theory has emerged in the last decade as an important topic of study. It appeared first in Deligne \cite{Del}, in relation with Voevodsky's solution of the Bloch-Kato conjecture. This was followed by the study of equvariant motivic spectra by Hu, Kriz, and Ormsby in \cite{hko}, who used them to study Karoubi's Hermitian $K$-theory and the motivic cobordism spectrum as $C_2$--spectra. Recently, there has been a serious push to develop solid foundations for equivariant motivic homotopy theory, in particular, by Heller, Krishna, Ormsby, Voineagu, {{\O}}stv\ae{}r and others (see for example \cite{CR}, 
\cite{HVO}, 
\cite{HKO15},
\cite{Her}, 
\cite{Hoy17}).

A significant difference between our work and that of the authors mentioned above is that they develop equivariant motivic homotopy theory starting from simplicial presheaves on $G$-schemes, rather than from presheaves of simplicial  $G$-sets on schemes, as we do. Motivic $G$-spaces in their sense can be seen as $G$-objects in motivic spaces, but the functor translating one to the other is not faithful, so the relationship between the two types of structure is somewhat delicate. Moreover the stabilizations we consider here are performed only with respect to motivic spaces equipped with trivial $G$-actions. Equivariant motivic homotopy theory built from simplicial presheaves on $G$-schemes, stabilized with respect to motivic spaces with potentially non-trivial actions, has the advantage of capturing more geometry.

In this paper we also consider only \emph{simplicial Galois extensions}, in the sense that the homotopy fixed points of a motivic $G$-space (or spectrum) $X$  that is nonequivariantly fibrant are computed as the fixed points of $\Hom (EG, X)$, where $EG$ is the geometric realization of an appropriate replacement of the usual simplicial $G$-set $E_{\bullet} G$. Unlike the situation in classical homotopy theory, $EG$, though a contractible free motivic $G$-space, is not the universal object with this property. Another contractible free motivic $G$-space does have the desired universal property, the \emph{\'etale} or \emph{geometric} $\mathbb{E} G$; one could hope to define \emph{geometric Galois extensions} in terms of homotopy fixed points computed with respect to $\mathbb{E} G$. 

An unpublished result of Heller (\fullref{prop:hell}) states that, within the right framework, $\KQ \to \KGL$ is a \emph{geometric} $C_2$--Galois extension, as long as the based scheme has characteristic different from $2$. We intend to develop a framework for \emph{geometric} Galois extensions of motivic spectra and to compute examples of such in future work. These computations will probably be harder than those presented here because the spectral sequences that we use in the simplicial case are not available in the geometric case. However, there should be rich geometric examples that we cannot capture in the simplicial framework. We discuss this  further in \fullref{sec:etalegalext}. Finally, in \fullref{sec:hopfgalext} we outline our plans to elaborate a theory of homotopic Hopf-Galois extensions of motivic ring spectra analogous to Rognes's theory for classical ring spectra \cite[Section 12]{rognes}.

\subsection*{Acknowledgments}
We thank Jeremiah Heller for extensive conversations and insight, and we thank both Jeremiah Heller and David Gepner for sharing drafts of their work in progress with us. We thank the Banff International Research Station for its hospitality and support during the Women in Topology II meeting, and we especially thank the organizers of WIT who made this project possible. 
This research was partially supported by the Simons Foundation, by the Mathematisches Forschungsinstitut Oberwolfach, and by the Hausdorff Institute for Mathematics in Bonn. This material is also based upon work supported by the National Science Foundation under Grant No. DMS-1606479 and Grant No. DMS-1612020.

\section{A formal homotopical framework for Galois theory}\label{sec:gal-formal}
We introduce here a formal framework in which to study a homotopical version of Galois theory, generalizing Rognes's Galois theory of commutative ring spectra \cite{rognes}.  In later sections we show that motivic spaces and motivic spectra both fit into this framework.

\begin{convention}\label{conv:modelcat1} Throughout this section we work in a pointed, symmetric monoidal model category $(\cat M, \otimes , I)$. We denote the internal hom functor by $\operatorname{Hom}$.  
\end{convention}

\subsection{Algebra in $\cat M$}

We refer to monoids and bimonoids in $\cat M$ as \emph{algebras} and \emph{bialgebras}, respectively.  We denote the category of all algebras by $\cat {Alg}$ and its full subcategory of commutative algebras by $\cat {CAlg}$. For any algebra $A$, the category of right $A$-modules is denoted $\cat {Mod}_{A}$ and, for any pair of $A$-modules $M$ and $N$, we denote by $\operatorname{Hom}_{A}(M,N)$ the obvious equalizer in $\cat M$.

For any bialgebra $(H,\mu, \Delta, \eta, \ve)$ in $\cat M$, the monoidal product on $\cat M$ induces a symmetric monoidal structure on the category ${}_{H}\cat {Mod}$ of left $H$-modules, given by 
$$(M, \lambda)\otimes (M', \lambda')=(M\otimes M', \widetilde \lambda),$$
where  $\widetilde \lambda$ is the composite 
$$H\otimes (M\otimes M' ) \xrightarrow {\Delta \otimes M\otimes M'} (H\otimes H) \otimes (M\otimes M' )\cong (H\otimes M)\otimes (H\otimes M') \xrightarrow {\lambda\otimes \lambda'} M\otimes M'.$$

Let  ${}_{H}\cat {Mod}_{A}$, ${}_{H}\cat {Alg}$, and ${}_{H}\cat {CAlg}$ denote the categories of $(H,A)$-bimodules, of $H$-algebras, i.e.  algebras in $({}_{H}\cat {Mod}, \otimes , I)$,  and of commutative $H$-algebras, respectively.

 A bialgebra $H$ in $\cat M$ is a \emph {Hopf algebra} if it admits an antipode, i.e., a morphism $\chi\colon H\to H$ in $\cat M$ such that the composite
 $$H \xrightarrow \Delta H\otimes  H \xrightarrow {H\otimes \chi} H\otimes H \xrightarrow \mu H$$
 is the identity.

The \emph{trivial (left) $H$-module} functor, denoted $\triv H\colon \cat M \to {}_{H}\cat {Mod}$  is defined by 
$\triv H X=(X, \ve\otimes X)$ for every object $X$ in $\cat M$.
Since $\triv HX \cong \operatorname{Hom}(I,X)\cong I\otimes X$, where $I$ is considered to be endowed with its trivial $H$-module structure, it admits both a left and a right adjoint. We call its right adjoint, $\operatorname{Hom}_{H}(I,-)$, the \emph{$H$-fixed points functor} and denote it  $(-)^{H}\colon {}_{H}\cat {Mod}\to \cat M$; the left adjoint, $I\otimes_{H} -$, is the \emph{$H$-orbits functor} denoted $(-)_{H}\colon {}_{H}\cat {Mod}\to \cat M$.   There are analogous adjunctions for right $H$-module structures, as well as induced fixed point and orbit adjunctions when $\cat M$ is replaced by $\cat {Mod}_{A}$, $\cat {Alg}$, and $\cat {CAlg}$.  

\begin{rem}\label{rem:fp-orb}
Recall that a right $H$-module structure on $M$ induces a left $H$-module structure on $\operatorname{Hom}(M,X)$, and vice-versa.
Observe that for any $H$-module $M$ and any object $X$ in $\cat M$, 
\begin{equation}
\label{eqn:fp-orb}
\operatorname{Hom}(M_{H},X)\cong \operatorname{Hom}(I\otimes _{H}M, X)\cong \operatorname{Hom}_{H}\big(I, \operatorname{Hom}(M,X)\big)\cong \operatorname{Hom}(M,X)^{H}.
\end{equation}
\end{rem}

\begin{convention}\label{conv:modelcat2} We suppose that  the categories $\cat {Mod}_{A}$, $\cat {Alg}$, $\cat {CAlg}$, ${}_{H}\cat {Mod}_{A}$, ${}_{H}\cat {Alg}$, $\cat {Alg}_{H}$, ${}_{H}\cat {CAlg}$, and $\cat {CAlg}_{H}$  admit model category structures such that, {for every algebra $A$, every bialgebra $H$, every algebra morphism $\vp\colon A\to B$, and every bialgebra morphism $\iota: K \to H$}, the following adjunctions are Quillen pairs
 $$\adjunction{\cat{Mod}_{A}}{\cat {Mod}_{B}}{-\otimes_{A}B}{\vp^{*}}\quad\text{and}\quad\adjunction{{}_{H}\cat{Mod}_{A}}{{}_{H}\cat {Mod}_{B}}{-\otimes_{A}B}{\vp^{*}} $$
 $$\adjunction{\cat{Mod}_{A}}{{}_{H}\cat{Mod}_{A}}{\triv H}{(-)^{H}}\quad\text{and}\quad \adjunction{{}_{H}\cat{Mod}_{A}}{{}_{K}\cat{Mod}_{A}}{\iota^{*}}{\operatorname{Hom}_{K}(H, -)}$$
 $$\adjunction{\cat{Alg}}{{}_{H}\cat {Alg}}{\triv H}{(-)^{H}}\quad\text{and}\quad \adjunction{\cat{Alg}_{H}}{\cat {Alg}}{(-)_{H}}{\triv H}$$
$$\adjunction{\cat{CAlg}}{{}_{H}\cat {CAlg}}{\triv H}{(-)^{H}}\quad\text{and}\quad \adjunction{\cat{CAlg}_{H}}{\cat {CAlg}}{(-)_{H}}{\triv H}$$
  $$\adjunction{{}_{H}\cat{Alg}}{{}_{K}\cat {Alg}}{\iota^{*}}{\operatorname{Hom}_{K}(H,-)}\quad\text{and}\quad \adjunction{{}_{H}\cat{CAlg}}{{}_{K}\cat {CAlg}}{\iota^{*}}{\operatorname{Hom}_{K}(H,-)},$$
and similarly when $H$ acts on the right. We assume also that every algebra $A$ is cofibrant as a module over itself, i.e., as an object in $\cat {Mod}_{A}$.
Finally, for every fibrant object $X$ in $\cat M$, the adjunction below is also a Quillen pair, where opposite categories are equipped with the opposite model structure:
$$\adjunction{\cat {Mod}_{H}^{op}}{{}_{H}\cat {Mod}}{\operatorname{Hom}(-,X)}{\operatorname{Hom}(-,X)}$$
 \end{convention}

\begin{rem} All $\triv H\dashv (-)^H$ and $(-)_H  \dashv \triv H$ adjunctions descend to adjunctions on the homotopy category, denoted $\htriv H\dashv \hfp H$ and $\horb H \dashv \htriv H$.
Furthermore,
for every right $H$-module $M$ and any object $X$ in $\ho\cat M$, 
 { \begin{equation}\label{eqn:hfp-horb}
\rhom{}(\Horb HM,X)\cong \operatorname{Hom} \big( (M^{c})_{H}, X^{f}\big) \cong \operatorname{Hom}(M^{c},X^{f})^{H}\cong \rhom{} (M,X)^{hH},
\end{equation}}
where $M^{c}$ denotes a cofibrant replacement in $\cat {Mod}_{H}$ and $X^{f}$ a fibrant replacement in $\cat M$.  The first isomorphism follows from the definition of left derived functors and the fact that $\cat M$ is a monoidal model category, while the second isomorphism is a special case of (\ref{eqn:fp-orb}).  The last isomorphism is a consequence of the fact that $\operatorname{Hom}(-,X^{f}): \cat {Mod}_{H}^{op}\to {}_{H}\cat {Mod}$ is a right Quillen functor, by \fullref{conv:modelcat2}.
\end{rem}
 
 In \fullref{sec:motmodstructures}, we show that model structures satisfying \fullref{conv:modelcat1} and \fullref{conv:modelcat2} exist when the underlying category is that of (pointed) motivic spaces or of motivic spectra.
 
\begin{rem}\label{rem:mult-hom} Note that for any algebra $A$, there is a multiplication on  $\operatorname{Hom}(H, A)$ given by the composite
$$\operatorname{Hom}(H, A)\otimes \operatorname{Hom}(H, A) \to \operatorname{Hom}(H\otimes H, A\otimes A)\xrightarrow {\operatorname{Hom}(\Delta, \mu)} \operatorname{Hom}(H, A),$$
where the first arrow is the transpose of evaluation on $H\otimes H$, and the second is determined by the comultiplication on $H$ and the multiplication on $A$. 

There is also a right $H$-action on $\operatorname{Hom}(H, A)$ given by the transpose of the composite
$$\operatorname{Hom}(H,A)\otimes H \otimes H \to \operatorname{Hom}(H,A) \otimes H \to A,$$
where the first arrow uses the multiplication in $H$, and the last arrow is the counit of the tensor-hom adjunction, {i.e., evaluation at $H$}.  More informally, we can think of $H$ as acting on $\operatorname{Hom}(H, A)$ by multiplication on the left in the domain; see also \fullref{lem:horb-free}. Note furthermore that $\operatorname{Hom}(H, A)^{H}\cong A$ for every algebra $A$, as the composite of the forgetful functor with $\triv H$ is obviously the identity. {This} implies that the composite of their {right} adjoints must be isomorphic to the identity.

On a similar note, for any object $X$ in $\cat M$, there is a natural algebra structure on $\operatorname{Hom}(X,X)$, given by ``composition,'' i.e., the transpose of the composite
$$\operatorname{Hom}(X,X)\otimes \operatorname{Hom}(X,X) \otimes X \to \operatorname{Hom}(X,X)\otimes X \to X,$$
where both arrows are built from the counit of the tensor-hom adjunction.
\end{rem}

\subsection{Derived structures}
In this section we make explicit the derived monoidal structure often implicity applied in stable homotopy theory, in particular by {Rognes} {in} \cite{rognes}.

The axioms of a monoidal model category imply that the homotopy category $\ho \cat M$ of $\cat M$ admits an induced closed symmetric monoidal structure, with tensor product $-\lot{} -$, internal hom $\rhom{} (-,-)$, and unit object $I^{c}$, a cofibrant replacement of the unit object in $\cat M$. The localization functor $\cat M \to \ho \cat M$ is monoidal, which implies that it preseves algebra and module structures. Henceforth in this paper, any expression involving $-\lot{}-$ or $\rhom{}(-,-)$ is to be considered as an object in $\ho \cat M$ and not in $\cat M$.

For any objects $X,Y$ in $\ho \cat M$, applying the functor $-\lot{} X$ to the component at $Y$ of the counit of the adjunction
$$\adjunction{\ho\cat M}{\ho\cat M}{-\lot{} X}{\rhom{}(X,-)}$$
produces a map
$$\rhom{}(X,Y)\lot{} X\lot{} Z \to Y\lot{}Z $$
for any object $Z$ in $\ho \cat M$. Taking the transpose of the composite
$$\rhom{}(X,Y) \lot{} Z \lot {} X\cong \rhom{}(X,Y) \lot{} X \lot {} Z \to Y\lot{} Z,$$
we obtain a map
$$\nu:\rhom{}(X,Y)\lot{} Z  \to \rhom{}(X,Y\lot{} Z),$$
which plays an important role in homotopical Galois theory.

Following {Lewis et al.} \cite{lewis-may}, as formulated by Rognes in \cite{rognes}, we introduce an appropriate notion of dualizability for objects in $\cat M$, when seen as objects in $\ho \cat M$.

\begin{defn}\label{defn:dualizable} 
The \emph{dual} of an object $X$ in $\ho \cat M$ is $DX=\rhom{}(X, I)$. An object $X$  is \emph{dualizable} if $\nu\colon DX \lot{} X \to \rhom{}(X,X)$ is an isomorphism.
\end{defn}

The proof of the lemma below can be found in \cite[\S III.1] {lewis-may}. 

\begin{lem}\label{lem:dualizable-properties} Let $X$ and $Z$ be objects in $\ho \cat M$.
\begin{enumerate}
\item If $X$ is dualizable, so is $DX$, and the canonical map $X \to DDX$ is an isomorphism.
\item If $X$ or $Z$ is dualizable, then $\nu:\rhom{}(X,Y)\lot{} Z  \to \rhom{}(X,Y\lot{} Z)$ is an isomorphism for every object $Y$.  In particular, if $X$ is dualizable, then 
$\nu:DX\lot{} Z  \to \rhom{}(X,Z)$ is an isomorphism for every object $Z$.
\end{enumerate}
\end{lem}

\begin{rem}\label{rem:dual-tensor}  
 \fullref{lem:dualizable-properties} implies that if $X$ is dualizable, then there is a natural isomorphism $D(X\lot{} X)\cong DX\lot{} DX$.
 \end{rem}

\begin{rem} If $A$ is a commutative algebra in $\cat M$, then, under our hypotheses on $\cat M$, the category $(\cat {Mod}_{A}, \otimes_{A}, A)$ is also a closed, symmetric monoidal category, with internal hom $\operatorname{Hom}_{A}(-,-)$, in which the definition and lemma above make sense.  Given an $A$-module $M$, we write $D_{A}M= \rhom A(M,A)$ and say that $M$ is \emph{$A$-dualizable} if $\nu\colon D_{A}M\lot A M \to \rhom A(M,M)$ is an isomorphism. Recall that the categories of left and right $A$-modules are isomorphic when $A$ is commutative, so that we could equally well work with left $A$-modules. 
\end{rem}

\begin{rem}\label{rem:cancellation} 
Recall that by \fullref{conv:modelcat2} any algebra $A$ is cofibrant as a module over itself.
It follows that there are natural isomorphisms $N\lot A A  \xrightarrow \cong N$  and $N\xrightarrow \cong \rhom A(A, N)$ for every right $A$-module $N$.  In particular, for any morphism of algebras $\vp: A \to B$, there are natural isomorphisms
$$\rhom B(B\lot A M, N)\cong \rhom B\big(B, \rhom A (M,N) \big) \cong  \rhom A(M, N)$$
for  all right $A$-modules $M$ and $N$.
\end{rem}

\begin{lem}\label{lem:extend-dualizable}  Let $\vp: A\to B$ be a morphism of commutative algebras.  If $M$ is a dualizable right $A$-module, then $M\lot A B$ is a dualizable $B$-module.
\end{lem}

\begin{proof}
The purely formal proof of Lemma 6.2.3 in \cite{rognes}, replacing smashes and function spectra by derived tensor and Hom, establishes this lemma, where  \fullref{rem:cancellation} provides the necessary isomorphisms.  
\end{proof}

\begin{defn} Let $A$ be an algebra in $\cat M$.  A left $A$-module $M$ is \emph{(homotopically) faithful} if whenever a right $A$-module $N$ is such that $N\lot A M\cong *$, it follows that $N\cong *$.  A morphism of algebras $\vp : A \to B$ is \emph{faithful} if $B$ is faithful as an $A$-module with respect to the module structure induced by $\vp$. 
\end{defn}

\begin{lem}\label{lem:faithful} Let $\vp: A\to B$ be a morphism of algebras, and let $M$ be a left $A$-module.
\begin{enumerate}
\item If $M$ is faithful as an $A$-module, then $B\lot A M$ is faithful as a $B$-module.
\item If $B$ is faithful as an $A$-module, and $B\lot A M$ is faithful as a $B$-module, then $M$ is faithful as an $A$-module.
\end{enumerate}
\end{lem}

\begin{proof} \fullref{rem:cancellation} implies that the proofs of Lemmas 4.3.3 and 4.3.4 in \cite{rognes} can be generalized to establish this result.
\end{proof}

\subsection{Norm maps}
We introduce in this section an abstract analogue of the norm map constructed in \cite[Part II, Section 5]{rognes} and prove that it is a weak equivalence under conditions analogous to those in \cite[Part II, Theorem 5.2.5]{rognes}.
  
Let $H$ be a bialgebra in $\cat M$, with multiplication $\mu$ and comultiplication $\Delta$.  Consider the following diagram of adjunctions.
$$ \xymatrix@R=4pc@C=4pc{ \ho\cat M \ar@{}[r]|{\perp} \ar@<1ex>[d]^{\htriv H} \ar@<1ex>[r]^{\htriv H} & \ho \cat {Mod}_{H}  \ar@<-1ex>[d]_{\htriv H} \ar@<1ex>[l]^{\hfp H} \\  
\ar@{}[r]|{\top} \ar@{}[u]|{\dashv} \ho{}_{H}\cat {Mod} \ar@<-1ex>[r]_{\htriv H} \ar@<1ex>[u]^{\horb H} &\ho( {}_{H}\cat {Mod}_{H}) \ar@<-1ex>[u]_{(-)_{hH}} \ar@<-1ex>[l]_{\hfp H}\ar@{}[u]|{\vdash} }$$ 
To be consistent with the approach in \cite{rognes}, fixed points are computed with respect to right $H$-actions and orbits with respect to left $H$-actions. Since 
$$(-)_{H} \circ \triv H =\triv H \circ (-)_{H}\colon {}_{H} \cat {Mod}\to  \cat {Mod}_{H},$$
it follows that 
$$\horb H \circ \htriv H =\htriv H \circ \horb H\colon  \ho({}_{H}\cat {Mod})\to  \ho (\cat {Mod}_{H}),$$
and therefore the theory of mates {(see Kelly and Street \cite{kelly-street})} implies that there exists a natural transformation 
\begin{equation}\label{eqn:kappa}\kappa\colon \horb H\circ \hfp H \Longrightarrow \hfp H\circ \horb H\colon \ho( {}_{H}\cat {Mod}_{H}) \to \ho \cat M.
\end{equation}

\begin{lem}\label{lem:horb-free} Let $H$ be a bialgebra in $\cat M$ such that the underlying object in $\cat M$ is cofibrant. For any object $X$ in $\ho \cat M$, 
$$ \Horb H {(X\lot {} H)} \cong X \cong \Hfp H{\rhom{}(H,X)}.$$
\end{lem}

\begin{proof}Consider the composable adjunctions
$$\xymatrix{\cat M\ar @<1.18ex>[rr]^{-\otimes H}&\perp&\cat M_{H}\ar @<1.18ex>[ll]^{\eta^{*}}\ar @<1.18ex>[rr]^{(-)_{H}}&\perp&\cat M\ar @<1.18ex>[ll]^{\triv H}}$$
where $\eta \colon I \to H$ is the unit of $H$, and so $\eta^\ast$ is the forgetful functor. Because $\eta^{*}\circ \triv H$ is the identity, the composite $(-)_{H}\circ (-\otimes H)$ is isomorphic to the identity, whence the composite
$$\xymatrix {\ho\cat M\ar[rr]^{-\lot{}H} && \ho\cat M_{H}\ar [rr]^{\horb H}&&\ho\cat M}$$
is also isomorphic to the identity.  We use here that the total left derived functor of $-\otimes H \colon \cat M \to \cat M_{H}$ is indeed $-\lot{} H$,  since $H$ is cofibrant.

Similarly, there are composable adjunctions
$$\xymatrix{\cat M\ar @<1.18ex>[rr]^{\triv H}&\perp&{}_{H}\cat M\ar @<1.18ex>[ll]^{(-)^{H}}\ar @<1.18ex>[rr]^{\eta^{*}}&\perp&\cat M\ar @<1.18ex>[ll]^{\operatorname{Hom}(H,-)}}$$
with $\eta^{*}\circ \triv H$ equal to the identity, whence the composite
$$\xymatrix {\ho\cat M\ar[rr]^{\rhom{}(H,-)} && \ho{}_{H}\cat M\ar [rr]^{\hfp H}&&\ho\cat M}$$
is isomorphic to the indentity. We use again that the total right derived functor of $\operatorname{Hom}(H, -)\colon \cat M \to {}_{H}\cat M$ is  $\rhom{}(H,-)$, since $H$ is cofibrant.
\end{proof}

Essential to the definition of the norm map is a sort of ``Poincar\'e duality'' result, generalizing \cite[Part II, Theorem 3.1.4]{rognes}, which we establish below as \fullref{prop:poincare}. A key element in the proof of this duality result is the following construction.

Let $H$ be a dualizable Hopf algebra such that the underlying object in $\cat M$ is cofibrant, from which it follows that $H\otimes H\cong H\lot{} H$; we  suppress this isomorphism henceforth, to simplify notation. Let 
$$\alpha \colon DH \lot{} H \to DH$$ be the right action of \fullref{rem:mult-hom}, in the derived closed monoidal structure on $\ho \cat M$.
The \emph{shear map} $\text{sh}\colon DH \lot{} H \to DH \lot{} H$ associated to $H$ is the composite 
$$DH \lot{} H \xrightarrow {DH \lot{}\Delta}  DH \lot{} H\lot{} H \xrightarrow {\alpha \lot {} H} DH \lot{} H.$$
It is easy to check (cf. \cite[Part II, Lemma 3.1.2(3)]{rognes}) that the shear map is a map of right $H$-modules, with respect to the ``free'' action
$$DH\lot{} H\lot {} H  \xrightarrow {DH\lot{} \mu} DH\lot {} H$$
in the source and the ``diagonal'' action
$$DH \lot{} H\lot{} H \xrightarrow{DH\lot{} H\lot{} (\chi\lot{} H)\Delta} DH \lot{} H\lot{} H \lot{} H\xrightarrow{(\alpha\lot{} \mu)\circ (324)} DH\lot {}H$$
in the target, where $(324)$ denotes the permutation $3\to 2 \to 4\to 3$.

Since $H$ admits an antipode $\chi$, the shear map is an isomorphism, with inverse 
$$(\alpha\lot {}ÊH) \circ (DH \lot{} \chi\lot{} H)\circ (DH \lot{} \Delta).$$
A straightforward computation shows that this map is indeed the inverse of the shear map (cf. \cite[Lemma 3.1.3]{rognes}).

The antipode $\chi$ on $H$ also enables us to define what Rognes calls \emph{actions through inverses} of $H$ on $DH$, as follows.  The left $H$-action through inverses on $DH$ is given by the composite
$$H\lot{} DH \xrightarrow \cong DH\lot{} H \xrightarrow {DH\lot{} \chi} DH \lot {}H \xrightarrow \alpha H,$$
where the first map is the symmetry isomorphism. The right $H$-action through inverses is defined similarly.

{\begin{exmp} Let $\cat M$ be the category $\cat {Ab}$ of abelian groups, seen as a model category with weak equivalences the isomorphisms, and let $G$ be a finite group. If $H$ is $\Z[G]$, then the shear map $\text{sh}$ sends $f(-)\otimes g$ in $D\Z[G] \otimes \Z[G]$ to $f(g-) \otimes g$. The left $\Z[G]$-action through inverses sends $g \otimes f(-)$ to $f(g^{-1}-)$.
\end{exmp}}

\begin{prop}\label{prop:poincare} If $H$ is a dualizable Hopf algebra in $\cat M$ such that the underlying object in $\cat M$ is cofibrant, then there is an isomorphism $DH \lot{} \Hfp HH \xrightarrow \cong H$ that is equivariant with respect to
\begin{itemize}
\item the left action through inverses on $DH$ and the usual left $H$-action on $\Hfp HH$ and $H$, and
\item the right action through inverses on $DH$, the trivial right action on $\Hfp HH$, and the usual right action on $H$.
\end{itemize}
\end{prop}

\begin{rem} The \emph{dualizing spectra} $\mathbb S^{adG}{={\mathbb S}[G]^{hG}}$ of \cite {rognes} are particular examples of the  object $\Hfp HH$, which we therefore refer to as the \emph{dualizing object} of $H$.
\end{rem}

{\begin{exmp} In $\cat {Ab}$ with $H=\Z[G]$ for a finite group $G$, this isomorphism is the map that sends $f$ in $D\Z[G] \otimes \Z[G]^G \cong D\Z[G]$ to $\sum_{g \in G} f(g)g$ in $\Z[G]$.
\end{exmp}}

\begin{proof}[{Proof of \fullref{prop:poincare}.}] The argument here is inspired by the proof of Theorem 3.1.4 in Part II of \cite{rognes}.  

Since the shear map is an isomorphism of right $H$-modules, it induces a isomorphism 
$$\text{sh}^{hH}: (DH \lot{} H)^{hH} \xrightarrow \cong (DH \lot{} H)^{hH}.$$
With respect to the $H$-action in the \emph{source} of the shear map, 
$$DH\lot {} \Hfp HH \cong \rhom{} (H, \Hfp HH)\cong \Hfp H{\rhom{} (H,H)}\cong \Hfp H{(DH\lot{} H)},$$
where the $H$-action on $DH$ (and thus on the source copy of $H$ in $\rhom{} (H,-)$) is taken to be trivial at each step.  
Finally, with respect to the diagonal $H$-action in the \emph{target} of the shear map, 
$$ \Hfp H{(DH\lot{} H)} \cong \Hfp H{\rhom{} (H,H)}\cong \rhom{} (\Horb HH, H)\cong \rhom{} (I, H)\cong H,$$
where the second isomorphism was established in (\ref {eqn:hfp-horb}).

See \cite[Part II, Theorem 3.1.4]{rognes} for a discussion of the required equivariance.
\end{proof}

 \begin{prop}\label{prop:sadg}  If $H$ is a dualizable Hopf algebra in $\cat M$ such that the underlying object in $\cat M$ is cofibrant, then for every left $H$-module $M$
 $$M\lot {} \Hfp HH \cong \Hfp H{(M\lot {} H)}$$ 
 as left $H$-modules.
 \end{prop}
 
\begin{proof} The argument here is inspired by the proof of Lemma 6.4.2 in \cite{rognes}.
 
There is a sequence of isomorphisms 
 \begin{align*}
 M\lot {} \Hfp HH &\cong \Hfp H{\rhom{}( H, M\lot {} \Hfp HH)}\\
 &\cong \Hfp H {(DH \lot{} M\lot {} \Hfp HH)}\\
 &\cong \Hfp H {(M\lot {}DH \lot{}  \Hfp HH)}\\
 &\cong \Hfp H{(M\lot {} H)}.
 \end{align*}
The first isomorphism follows from \fullref{lem:horb-free}, the second from the dualizability of $H$, and the last from \fullref{prop:poincare}.
 \end{proof}

The definition, theorem, and proof below are inspired by \cite[Part II, Definition 5.2.2 and Theorem 5.2.5]{rognes}.
 
 \begin{defn}  Let $H$ be a dualizable Hopf algebra such that the underlying object in $\cat M$ is cofibrant, and let $M$ be a {right $H$-module in $\cat M$}.  
 The \emph{norm map} for $M$ is the composite
 $$\xymatrix{\Horb H{(M\lot{}\Hfp HH)}\ar[r]^(0.45){\cong}& \Horb H {\big(\Hfp H{(M\lot {} H)}\big)} \ar[r]^(0.5){\kappa}& \Hfp H {\big(\Horb H{(M\lot {} H)}\big)}\ar [r]^(0.65){\cong}& \Hfp HM,}$$
 where the first isomorphism is a consequence of \fullref{prop:sadg}, the middle map is a component of the natural transformation (\ref{eqn:kappa}), and the last isomorphism follows from an ``untwisting'' lemma of {Lewis et al.} \cite[p. 76]{lewis-may}, which implies that $(M\lot{} H)_{hH}\cong M\lot{} H_{hH}\cong M$, equivariantly with respect to the right $H$-action.  
 \end{defn}
 
 \begin{thm}\label{thm:norm}  Let $H$ be a dualizable Hopf algebra such that the underlying object in $\cat M$ is cofibrant. If $M=H\lot{} X$ is the left $H$-module induced up from some $X\in \operatorname{Ob}\ho\cat M$, then the associated norm map
 $$\Horb H{(M\lot{}\Hfp HH)} \to \Hfp HM$$
 is an isomorphism.
 \end{thm}
 
 \begin{proof} An ``untwisting'' argument shows that the source of the norm can be simplified as
 $$(H\lot{} X \lot{} \Hfp HH)_{hH}\cong X \lot{} \Hfp HH,$$ 
 while the target of the norm map is
 $$\Hfp H {(H\lot{} X)} \cong \Hfp H{(DH \lot{} \Hfp HH \lot{} X)}\cong \Hfp H{\rhom{}(H, X\lot{} \Hfp HH)} \cong X\lot {} \Hfp HH.$$
 One can check that the norm map respects these identifications.
 \end{proof}

\subsection{Homotopical Galois theory}

Generalizing Definition 4.1.3 in \cite{rognes}, we can formulate a notion of homotopical Galois extensions in a pointed, symmetric monoidal model category $\cat M$ satisfying \fullref{conv:modelcat1} and \fullref{conv:modelcat2} as follows.  We begin by chararcterizing the sort of objects we study.

\begin{defn}\label{defn:galois-data} \emph{Galois data} in the monoidal model category $(\cat M, \otimes , I)$ consist of 
\begin{itemize}
\item a dualizable Hopf algebra $H$,
\item an algebra $A$,
\item an $H$-algebra $B$, where $H$ acts on the left, and
\item a morphism $\vp:  \triv{H}{A}\to B$ of $H$-algebras,
\end{itemize}
which we denote $\vp:  \triv{H}{A}\to B^{\circlearrowleft H}$.   If $A$ and $B$ are both commutative algebras, then we say that $\vp:  \triv{H}{A}\to B^{\circlearrowleft H}$ is \emph{commutative Galois data}.

More generally, any {such} morphism in $\ho ({}_{H}\cat {Alg})$, where $H$ is a dualizable Hopf algebra, is \emph{homotopical Galois data}. 
\end{defn}

 Recall from above the trivial action-fixed points adjunction 
 $$\adjunction{\cat{Alg}}{{}_{H}\cat {Alg}}{\triv H}{(-)^{H}},$$
 which induces an adjunction
 $$\adjunction{\ho\cat{Alg}}{\ho {}_{H}\cat {Alg}}{\htriv H}{\hfp {H}},$$
 where $\hfp H$ is called the \emph{homotopy fixed points} functor.  Since $\triv H$ preserves all weak equivalences, $\Hfp H {(\triv HA)}\cong A$ in $\ho \cat {Alg}$ for all algebras $A$; the analogous result holds for commutative algebras as well. 
 
Given (commutative) Galois data $\vp:  \triv{H}{A}\to B^{\circlearrowleft H}$ (strict or homotopical),  let  $$\vp ^{hH}: A \to B^{hH}$$ in $\ho \cat {Alg}$ (or $\ho \cat {CAlg}$) denote the image of $\vp$ under the homotopy fixed points functor.   Another important morphism associated to $\vp$, 
$$\beta_{\vp}: B\lot{A} B \to \rhom{}(H,B)$$ 
in $\ho \cat M,$ 
is defined to be the transpose of
$$B\lot {A} (B\lot{} H) \xrightarrow {B\lot{A} \rho_{B}} B\lot {A} B \xrightarrow {\bar \mu} B,$$
where $\bar \mu$ denotes the morphism induced by the multiplication map of $B$, which can be viewed as an $A$-algebra via $\vp$.

\begin{defn}\label{defn:galois-ext}   A \emph{homotopical $H$-Galois extension} consists of (strict or homotopical) Galois data $$\vp:  \triv{H}{A}\to B^{\circlearrowleft H}$$ such that both of the morphisms
  $\vp ^{hH}: A \to B^{hH}$ and $\beta_{\vp}: B\lot{A} B \to \rhom{}(H,B)$ 
are isomorphisms.
\end{defn}

\begin{rem}Since the localization functor $\cat M \to \ho \cat M$ sends weak equivalences to isomorphisms, being a homotopical Galois extension in $\cat M$ is invariant under weak equivalences in the source and target.
\end{rem}

\begin{rem}\label{rem:exab}
The definition above generalizes the classical definition of a $G$-Galois extension of rings, where $G$ is a group. Rings are algebras in the category $\cat {Ab}$ of abelian groups, equipped with the usual tensor product, and the group ring $\mathbb Z[G]$ is naturally a bialgebra in  $\cat {Ab}$.  Actions of $G$ on a ring $R$ correspond to $\mathbb Z[G]$-module structures on $R$.  If $\cat {Ab}$ is seen as a model category in which the weak equivalences are the isomorphisms, then a homotopical $\mathbb Z[G]$-Galois extension of rings  $\triv{\mathbb Z[G]}{A}\to B^{\circlearrowleft \mathbb Z[G]}$ is exactly a $G$-Galois extension.  

Similarly, a $G$-Galois extension of ring spectra, in the sense of Rognes \cite [Definition 4.1.3]{rognes}, is exactly a homotopical $\mathbb S[G]$-Galois extension, where $\mathbb S[G]$ denotes the group-ring spectrum. 
\end{rem}

\begin{exmp}\label{exmp:trivial-ext}  Let $H$ be a dualizable Hopf algebra  in $\cat M$ that is cofibrant as an object in $\cat M$. For every algebra $A$, there is Galois data  $$\tau_{A}:\triv H A \to \operatorname{Hom}(H, A)^{\circlearrowleft H},$$ called the \emph{trivial extension of $A$}, induced by the counit $\ve: H\to I$.    By \fullref{lem:horb-free}, $\rhom{}(H, A)^{hH}\cong A$, thus
$$\big(\tau_{A}: \triv H A\to \rhom{}(H, A)\big)^{hH}\cong (\id_{A}:A \to A).$$

As for $ \beta_{\tau_{A}}$, observe that since $H$ is dualizable, there are isomorphisms
\begin{align*}
\rhom{} \big(H, \rhom{}(H,A)\big) &\cong DH \lot{} \rhom{}(H,A)\\
&\cong DH \lot{} A\lot {A}\rhom{}(H,A)\\
&\cong \rhom{}(H,A)\lot A \rhom{} (H,A).
\end{align*}
It follows that $\beta_{\tau_{A}}$ is also an isomorphism and thus that $\tau_{A}$ is a homotopical Galois extension.
\end{exmp}

\subsection{Dualizability and Galois extensions}

Here we establish formal generalizations of results from \cite{rognes} that clarify the relationship between dualizability and Galois extensions.

Recall that the localization map $\gamma:\cat M \to \ho\cat M$ is monoidal and therefore sends algebras and modules in $(\cat M, \otimes, I)$ to algebras and modules in $(\ho \cat M, \lot{}, I^{c})$.  If a coalgebra $C$ in $\cat M$ is cofibrant as an object in $\cat M$, then its image under $\gamma$ is a coalgebra in $\ho \cat M$, since $C\otimes C \cong C\lot{} C$ in $\ho \cat M$.  In particular, if $H$ is a bialgebra in $\cat M$ that is cofibrant as an object in $\cat M$, then its image in $\ho \cat M$ is also a bialgebra.

\begin{defn} Let $\vp:  \triv{H}{A}\to B^{\circlearrowleft H}$ be Galois data, where $H$ is cofibrant as an object in $\cat M$.  The \emph{twisted algebra} $B\langle H\rangle$  in $\ho \cat M$ consists of the object $B\lot{} H$, equipped with the multiplication determined as a free left $B$-module and free right $H$-module map by the composite
$$H\lot{} B\xrightarrow {\Delta \lot{} B} H \lot{} H \lot{} B \xrightarrow {H \lot{} \lambda}  H  \lot{} B \cong B  \lot{} H,$$
i.e., apply $B\lot{} -$ on the left and $-\lot{} H$ on the right to the composite above, then postcompose with multiplication in $B$ on the left and in $H$ on the right, to land in $B\lot{} H$.
\end{defn}

\begin{exmp}
If $\cat {Ab}$ is as in \fullref{rem:exab} and $H = \Z[G]$, so that $B$ is a $G$-module, then $B\langle H\rangle$ is the twisted group ring with coefficients in $B$. Its underlying abelian group is $B[G]$, with multiplication determined by $(b_1[g_1])(b_2[g_2]) = b_1 g(b_2)[g_1g_2]$.
\end{exmp}

Note that $\vp\lot{} \eta: A \to B\langle H\rangle$ is a morphism of algebras and therefore induces right and left $A$-actions on $B\langle H\rangle$. Moreover, $B\langle H\rangle$ is naturally a left $B$-module, as well as a left $H$-module, since it is a tensor product of left $H$-modules. 

There is a morphism of $A$-algebras
$$j_{\vp}:B\langle H\rangle \to \rhom{A}(B,B)$$
given by taking the transpose of the composite
$$(B\lot{} H)\lot{A} B\to  B\lot{A}B \to B,$$
where the first arrow is induced by the left action of $H$ on $B$, 
and the second is given by the multiplication in $B$.    The morphism $j_{\vp}$ is a morphism of left $B\langle H \rangle$-modules, since it is a morphism of algebras, and therefore of left $B$-modules and left $H$-modules.

\begin{lem}\label{lem:j} Let $\vp\colon  \triv{H}{A}\to B^{\circlearrowleft H}$ be Galois data, where $H$ is dualizable, and the object in $\cat M$ underlying $H$ is cofibrant. If $\beta_{\vp}: B\lot{A} B \to \rhom{}(H,B)$  is an isomorphism, then

\begin{enumerate}
\item for each right $B$-module $M$, there is a natural isomorphism $$\beta_{\vp,M }: M\lot{A} B \to \rhom{}(H,M);$$
\item the natural map $j_{\vp}:B\langle H\rangle \to \rhom A(B,B)$ is an isomorphism; and
\item  for each right $B$-module $M$, there is a natural isomorphism  
$$j_{\vp, M}: M\lot{} H\to \rhom A(B,M).$$
\end{enumerate}
\end{lem}

\begin{proof}  The proof of (1) is essentially identical to that of Lemma 6.1.2 (a) in \cite{rognes}, once we replace smashing and function spectra by derived tensors and homs.  Statement (2) is a special case of (3), of which the proof follows the lines of that of Lemma 6.1.2 (c) in \cite{rognes}, again replacing smashing and function spectra by derived tensors and homs and weak equivalences by isomorphisms.  In particular, $j_{\vp, M}$ is obtained by taking the transpose of the composite
$$(M\lot{} H)\lot{A} B\to  M\lot{A}B \to M,$$
where the first arrow is induced by the left action of $H$ on $B$, and the second is given by the right $B$-action. 
\end{proof}

For all Galois data $\vp\colon  \triv{H}{A}\to B^{\circlearrowleft H}$ and all cofibrant right $A$-modules $M$, there is a useful, natural map $\omega: M\lot A B^{hH}\to (M\lot A B)^{hH}$ in $\ho \cat M$, constructed as follows.  Let 
$$M\otimes _{A}-\colon A\negthinspace \downarrow \negthinspace \cat {Alg} \to \cat M$$
denote the functor defined on objects by sending $\vp\colon A \to B$ to  $M\otimes_{A} B$, where $B$ is endowed with the left $A$-module structure induced by $\vp$.  We define 
$$M\otimes _{A}-\colon \triv HA\negthinspace \downarrow \negthinspace {}_{H}\cat {Alg}  \to {}_{H}\cat {Mod}$$
similarly.
Observe that the composite
$$\xymatrix{A\negthinspace \downarrow \negthinspace \cat {Alg}\ar [rr]^(0.4){\triv H}&&\triv HA\negthinspace \downarrow \negthinspace {}_{H}\cat {Alg} \ar [rr]^(0.6){M\otimes_{A}-}&&{}_{H}\cat {Mod}}$$
is equal to the composite
$$\xymatrix{A\negthinspace \downarrow \negthinspace \cat {Alg}\ar [rr]^{M\otimes _{A}-}&&\cat M \ar [rr]^{\triv H}&&{}_{H}\cat {Mod},}$$
from which it follows that 
$$(M\lot A -)\circ \htriv H =\htriv H \circ (M\lot A -),$$
where we deduce from the cofibrancy of $M$ that $M\lot A-$ is indeed the {total} left derived functor of $M\otimes _{A}-$.
The theory of mates  \cite{kelly-street} then implies that there is a natural transformation
$$\omega: M\lot A \Hfp H{(-)} \Longrightarrow \Hfp H {(M\lot A-)}\colon \ho (\triv HA\negthinspace \downarrow \negthinspace {}_{H}\cat {Alg}) \to \ho\cat M.$$

\begin{lem}\label{lem:psi} Let $\vp:  \triv{H}{A}\to B^{\circlearrowleft H}$ be Galois data, and let $M$ be a cofibrant right $A$-module.  If
\begin{enumerate}
\item $\vp$ is a homotopical Galois extension, or
\item $M$ is dualizable as an $A$-module,
\end{enumerate}
then the natural map
$$\omega: M\lot A B^{hH}\to (M\lot A B)^{hH}$$
is an isomorphism.
\end{lem}

\begin{proof} The proofs of  \cite[Lemma 6.1.3]{rognes} and \cite[Lemma 6.2.6]{rognes} generalize easily to this context, for the proofs of (1) and (2), respectively, where we replace occurences of ``weak equivalence'' with ``isomorphism.''
\end{proof}

\begin{prop}\label{prop:galois-dualizable} If $\vp:  \triv{H}{A}\to B^{\circlearrowleft H}$ is a homotopical Galois extension, where the objects in $\cat M$ underlying $B$ and $H$ are cofibrant, then $B$ is a dualizable $A$-module.
\end{prop}

\begin{proof} The argument in the proof of \cite[Proposition 6.2.1]{rognes} works in this case, with the usual replacements, thanks to \fullref{thm:norm}, \fullref{lem:j}, and \fullref{lem:psi}.  
\end{proof}

We conclude this section with a useful, alternate characterization of faithful derived Galois extensions.
\begin{prop}\label{prop:characterize} If $\vp:  \triv{H}{A}\to B^{\circlearrowleft H}$ is Galois data, where the objects in $\cat M$ underlying $B$ and $H$ are cofibrant, then $\vp$ is a faithful homotopical Galois extension if and only if $\beta_{\vp}$ is an isomorphism, and $B$ is faithful and dualizable as an $A$-module. 
\end{prop}

\begin{proof} The argument in \cite[Proposition 6.3.2]{rognes} works in this case, with the usual replacements, thanks to \fullref{prop:galois-dualizable} and part (2) of \fullref{lem:psi}.
\end{proof}

\subsection{Invariance under cobase change}\label{sec:cobasechange}
Generalizing results of Rognes \cite[Section 7]{rognes}, we now formulate and sketch proofs of invariance results that play an important role in our proof of the forward Galois correspondence for homotopical Galois extensions.  Throughout this section we suppose that $H$ is a dualizable Hopf algebra such that the underlying object in $\cat M$ is cofibrant.

\begin{lem}   Let $\vp:  \triv{H}{A}\to C^{\circlearrowleft H}$ be (strict or homotopical) commutative Galois data, and let $\psi: A \to B$ be  a morphism of commutative algebras.  If $\vp$ is a homotopical Galois extension, and
\begin{enumerate}
\item $C$ is faithful as an $A$-module, or
\item $B$ is dualizable as an $A$-module,
\end{enumerate} 
then the induced algebra map $\bar \vp:  \triv{H}{B}\to B\lot A C^{\circlearrowleft H}$ is also a homotopical Galois extension.  Moreover if (1) holds, then $B\lot A C$ is faithful as an $A$-module.
\end{lem}

\begin{proof} See the proofs of \cite[Lemmas 7.1.1 and 7.1.3]{rognes}.  We apply \fullref{prop:characterize}, the hypotheses of which hold thanks to \fullref{lem:extend-dualizable}, \fullref{lem:faithful}, and \fullref{prop:galois-dualizable}, together with the facts that $\vp$ is {a} homotopical Galois {extension} and that $H$ is dualizable.  
\end{proof}

\begin{lem}\label{lem:reflect-galois}  Let $\vp:  \triv{H}{A}\to C^{\circlearrowleft H}$ be (strict or homotopical) commutative Galois data, and let $\psi: A \to B$ be  a morphism of commutative algebras  such that $B$ is faithful and dualizable over $A$.  If the induced algebra map $\bar \vp:  \triv{H}{B}\to B\lot A C^{\circlearrowleft H}$ is a homotopical Galois extension, then so is $\vp$.  Moreover, if $\bar \vp $ is faithful, then so is $\vp$.
\end{lem}

\begin{proof} See the proof of \cite[Lemma 7.1.4]{rognes}.  We apply \fullref{lem:faithful} and part (2) of \fullref{lem:psi}, together with the faithfulness and dualizability of $B$ over 
$A$.
\end{proof}

\subsection{From subgroups to subextensions}\label{sec:subtoext}
Motivated by \cite[Theorem 7.2.3]{rognes}, we now establish a forward Galois correspondence for homotopical Galois extensions. Since Rognes's proof of the backward Galois correspondence for commutative ring spectra  \cite[Theorem 11.2.2]{rognes} is far from formal, requiring clever application of Goerss-Hopkins obstruction theory, we do not expect to be able to prove an analogous result in an arbitrary monoidal model category.

\begin{defn} A morphism $\iota\colon K\to H $ of bialgebras in $\cat M$ is \emph{allowable} if $\iota $ extends to an isomorphism $K\lot{}\Horb KH \cong H $ in $\ho {}_{K}\cat {Mod}$.  An allowable morphism $\iota\colon K\to H$ is \emph{normal} if $\Horb KH$ admits an algebra structure such that the counit $\ve: K\to I$ induces an algebra map $H\to \Horb KH$.
\end{defn}

\begin{exmp} The definition above generalizes the definition of allowable subgroup from \cite[Section 7.2]{rognes}.  In particular, if $G$ is a finite group, and $K$ is any subgroup of $G$, then the ring spectrum map $\mathbb S[K] \to \mathbb S[G]$ determined by the inclusion of $K$ into $G$ is allowable.  It is normal if $K$ is a normal subgroup of $G$.

Similarly, for any subgroup $K$ of a finite group $G$, the inclusion $\Z[H] \to \Z [G]$ is an allowable morphism of bialgebras in $\cat{Ab}$.
\end{exmp}

Recall that according to \fullref{conv:modelcat1}, a morphism of bialgebras $\iota\colon K\to H$ induces a right Quillen functor $\iota^{*}\colon {}_{H}\cat {CAlg} \to {}_{K}\cat{CAlg}$.    Let $\mathbb R \iota^{*}$ denote the associated total right derived functor on homotopy categories.  Since $\iota^{*}\circ \triv H =\triv K$, it follows that $\mathbb R\iota^{*}\circ \htriv H =\htriv K$ and thus, by the theory of mates, there is a natural transformation
$$\zeta: \hfp H \Longrightarrow \hfp K \circ \mathbb R\iota ^{*}.$$
For any commutative Galois data $\triv HA \to B^{\circlearrowleft H}$, {there is therefore a map
$$A\cong \Hfp HB \xrightarrow {\zeta _{B}}\Hfp K {\big( \mathbb R\iota^{*}(B)\big)}$$
in $\ho \cat {CAlg}$. 
Further, the counit of the $\htriv K \dashv \hfp K$ adjunction gives a map
$$\htriv K \Hfp K {\big( \mathbb R\iota^{*}(B)\big)} \to \mathbb R \iota^{*}(B)$$
in $\ho {}_{K}\cat {CAlg}$.}

\begin{thm}\label{thm:galois-corresp} Let $H$ and $K$ be dualizable Hopf algebras such that the underlying objects in $\cat M$ are cofibrant. Let $\vp\colon\triv HA \to B^{\circlearrowleft H}$ be a faithful homotopical Galois extension of commutative algebras such that the object underlying $B$ in $\cat M$ is cofibrant.  If $\iota\colon K\to H$ is an allowable morphism, then 
$$\htriv K \Hfp K{\big(\mathbb R \iota^{*}(B) \big)} \to \mathbb R \iota^{*}(B)^{\circlearrowleft K}$$ 
is a homotopical $K$-Galois extension.  If $\iota$ is normal, then 
$$A \to \Hfp K{\big(\mathbb R \iota^{*}(B) \big)}$$
is a homotopical $\Horb KH$-Galois extension.
\end{thm}

\begin{proof} Our proof here is inspired by that of  \cite[Theorem 7.2.3]{rognes}.

To simplify notation, let $C= \Hfp K{\big(\mathbb R \iota^{*}(B) \big)}$.  We also abuse notation slightly and write $B$ instead of $\mathbb R \iota^{*}(B)$ and $A$ instead of $\Hfp HB$. Consider the following commutative diagram in $\ho \cat M$.
$$\xymatrix{B\ar [r]& B\lot A B \ar [r]^(0.4){\beta_{\vp}}_(0.4){\cong}& \rhom{}(H,B) & \rhom {}(H,B)\ar [l]_{=}\\
C\ar [u]^{\ve}\ar [r]&B\lot AC\ar[u]\ar [r]^(0.4){\hat \beta_{\vp}}_(0.4){\cong} & \Hfp K{\rhom{}(H,B)}\ar [u]& \rhom{}(\Horb KH, B)\ar [l]_{\upsilon}^{\cong}\ar [u]\\
A\ar[u]^{\zeta}\ar [r]^{\vp}&B\ar[u]\ar [r]^(0.4){=}&B\ar[u]&B\ar[u]\ar [l]_{=}}$$ 
The lefthand squares are pushouts.  The map $\beta_{\vp}$ is an isomorphism, since $\vp $ is a homotopical Galois extension, while the map $\hat\beta_{\vp}$ is equal to the composite
$$B\lot AC \xrightarrow {\omega} \Hfp K{(B\lot A B)} \xrightarrow {\Hfp K{\beta_{\vp}}} \Hfp K{\rhom{}(H,B)}.$$
By part (1) of \fullref{lem:psi}, the map $\omega$ is an isomorphism, whence $\hat\beta_{\vp}$ is as well.  Equation (\ref{eqn:hfp-horb}) implies that the natural map $\upsilon$ is also an isomorphism. 

If $\iota$ is allowable, then
$$\rhom{}(H, B)\cong \rhom{}(K\lot {} \Horb KH, B) \cong \rhom{}\big(K, \rhom{}(\Horb KH, B)\big),$$
whence the upper righthand vertical map
$$ \rhom{}(\Horb KH, B)\to \rhom{}(H, B)$$
is a homotopical $K$-Galois extension, as seen in \fullref{exmp:trivial-ext}.  It follows that 
$$B\lot AC \to B\lot A B$$
is also homotopical $K$-Galois.  Since $B$ is  dualizable over $A$ by \fullref{prop:galois-dualizable} and faithful over $A$ by hypothesis, \fullref{lem:extend-dualizable} and \fullref{lem:faithful} imply that $B\lot AC$ is dualizable and faithful over $C$.  By \fullref{lem:reflect-galois}, we conclude that $\ve\colon C \to B$ is a homotopical $K$-Galois extension.

If $\iota$ is  normal, then \fullref{exmp:trivial-ext} implies that $B\to \rhom{}(\Horb KH, B)$ is homotopical $\Horb KH$-Galois.  Since $B$ is  dualizable over $A$ by \fullref{prop:galois-dualizable} and faithful over $A$ by hypothesis, \fullref{lem:reflect-galois} implies that $\zeta\colon A\to C$ is a homotopical $\Horb KH$-Galois extension, as desired.
\end{proof}

\section{Motivic model structures}\label{sec:motmodstructures}

In this section we establish the existence of model category structures on categories of motivic spaces and spectra to which the formal Galois theory framework of the previous section applies.  We make extensive use of the theory of left- and right-induced model category structures, recalled in detail in \fullref{appendix}.  We refer the reader to \fullref{glossary} for a table presenting all of the numerous model structures constructed in this section.

\subsection{Simplicial presheaves}
Let $\CC$ be a small category, and let $\sPre{\CC}$ denote the category of simplicial presheaves on $\CC$, i.e., functors from $\CC^{op}$ to  $\sSet$, or equivalently, simplicial objects in the category  $\Pre{\CC}$ of set-valued presheaves on $\CC$. Every simplicial set $A$ can be viewed as a constant presheaf, which by slight abuse {of notation} we also denote by $A$. Moreover, if $X $ is an object of $\CC$, we also denote by $X$ the simplicial presheaf it represents, i.e., its image under the Yoneda embedding $\CC \lra \sPre{\CC}$ (constant in the simplicial direction).

Let $\psSet$ be the category of pointed simplicial sets, and denote by $\psPre{\CC}$  the category of pointed simplicial presheaves, i.e.,  functors $F: \CC^{op} \lra \psSet$. 
When we work with pointed simplicial presheaves we compose the Yoneda embedding $\CC \lra \sPre{\CC}$ with the functor adding a disjoint basepoint $\sPre{\CC} \lra \psPre{\CC}$ to get an embedding $\CC \lra \psPre{\CC}$. The category $\psPre{\CC}$ is enriched over $\psSet$.

 On the other hand, if $\CC$ itself has a terminal object $*$, one can also embed the category $\CC_{*}= *\negthinspace\downarrow\negthinspace \CC$ of pointed objects in $\CC$ into $\psPre{\CC}$.  Under the Yoneda embedding, an object $c\colon *\to C$ in $\CC_{*}$   can be seen as a map of simplicial presheaves.  Since the simplicial presheaf associated to $*$ is the terminal object in $\psPre{\CC}$, it follows that the unpointed Yoneda embedding induces a pointed Yoneda embedding.

The categories $\sPre{\CC}$ and $\psPre{\CC}$ can be equipped with several well known model structures, such as the projective, injective, and flasque model structures (see, for example, Isaksen \cite[Theorem 2.2 and 3.7]{isaksen-flasque}). The identity functor induces Quillen equivalences between any pair {among} these three model structures.  For the purposes of this paper, we use the injective model structure, in which cofibrations and weak equivalences are both defined objectwise. The injective model structure is left proper, cellular, and simplicial {(see Lurie \cite[Proposition A.2.8.2 and Remark A.2.8.4]{HTTLurie})}. 

There is a closed monoidal structure on $\sPre{\CC}$ (respectively, $\psPre{\CC}$) given by  the objectwise product  in $\sSet$ (respectively, smash product in $\psSet$). Since the monoidal structure is defined objectwise, algebras, coalgebras, modules, and comodules in $\sPre{\CC}$ and $\psPre{\CC}$ are also defined objectwise, e.g., $A$ is an algebra in $\sPre{\CC}$ if for every $X \in \CC$, $A(X)$ is a simplicial monoid. Together with the objectwise monoidal structure, the injective model structure on $\sPre{\CC}$ (respectively, $\psPre{\CC}$) forms a monoidal model category, since $(\sSet, \times , *)$ (respectively, $(\psSet, \wedge , S^{0})$) is a monoidal model category.

\subsection{Motivic  spaces}
Let $S$ be a Noetherian scheme of finite Krull dimension, and let $\Sm S$ be the category of separated, smooth schemes of finite type over $S$, which we simply call smooth schemes over $S$.
To do motivic homotopy theory, one first builds the model category of motivic spaces, of which the underlying category is the category $\sPre{\Sm{S}}$. Embedding $\Sm{S}$ into $\sPre{\Sm{S}}$ formally adjoins colimits to $\Sm{S}$, which is far from cocomplete. 
 
The injective, projective, and flasque model structures on $\sPre{\Sm{S}}$, which are referred to in this context as \emph{global} model structures, have the drawback that colimits 
in $\Sm{S}$ may not be preserved under the Yoneda embedding. To repair this problem, one localizes the global model category structure with respect to a well chosen Grothendieck topology on $\Sm{S}$. {Usually, one chooses the Nisnevich topology, as this leads to the representability of important motivic invariants, such as $K$-theory.} The resulting model category structure is called a \emph{local} model structure on $\sPre{\Sm{S}}$.

More precisely, the \emph{local injective model structure} on $\sPre{\Sm{S}}$ is the left Bousfield localization of the global injective model structure at the class of all Nisnevich hypercovers \cite[{Definition} 4.1]{isaksen-flasque}. A more direct description of this localization can be given in terms of \emph{elementary distinguished Nisnevich squares}, i.e., cartesian diagrams of schemes
\[\xymatrix{
U \times_X V \ar[r] \ar[d]  & V \ar[d]^-p\\
U \ar[r]^-i & X
} \]
such that $i$ is an open immersion, $p$ is \'etale, and $p^{-1}\big(X \smallsetminus i(U)\big) \to X \smallsetminus i(U)$ is an isomorphism on the induced reduced schemes. The last property ensures that the square is also co-cartesian.  The local injective model structure is then the left Bousfield localization of the global injective model structure on $\sPre{\Sm{S}}$ at the set of maps
\[ U \coprod_{U \times_X V} V \to X \]
ranging over all elementary distinguished squares \cite[Thm. 4.9]{isaksen-flasque}.

Finally, we perform one more left Bousfield localization of the local injective model structure, namely at the maps $X \times \A^1 \to X$ for all $X\in \Sm{S}$.  (In the localized structure, the affine line {thus} plays the role of the unit interval in ordinary homotopy theory.) {We call the resulting model category, denoted $\spc$, the \emph{category of motivic spaces} with the \emph{motivic injective model structure}. Its distinguished classes of maps are called motivic equivalences, motivic fibrations, and motivic cofibrations. Note that the underlying category is just $\sPre{\Sm{S}}$. By a similar process, we construct the model category $\pspc$ of pointed motivic spaces.}

\begin{lem}\label{lem:cat_mot}
The model categories $\spc$  and $\pspc$ are simplicial, left proper, and cellular. Moreover, they are closed symmetric monoidal and satisfy the monoid axiom.
\end{lem}

\begin{proof}
The first three conditions are satisfied because these properties hold for the global injective model structure and are preserved by left Bousfield localization. The pushout product axiom follows from {Pelaez} \cite[Corollary 2.3.5]{slice} and immediately implies the monoid axiom, since every object is cofibrant.
\end{proof}

Thanks to this lemma, the next result is an immediate consequence of {Schwede and Shipley} \cite[Theorem 4.1(3)]{schwede-shipley}.
 
\begin{cor}\label{cor:algright}
There exists a model category structure on the category $\aspreS$ of algebras in $\sPre{\Sm{S}}$
right-induced from $\spc$  by the adjunction
$$\adjunction{\sPre{\Sm S}}{\aspreS,}{F}{U}$$
where $F$ denotes the free associative algebra functor and $U$ the forgetful functor.
\end{cor}

We denote this model category structure by $\aspc$ and call it the \emph{category of motivic algebras}.

Hornbostel's \cite[Theorem 3.12]{Hornbostel} implies the following result.
\begin{lem}\label{cor:calgright}
There exists a model category structure on the category $\caspreS$ of commutative algebras in $\sPre{\Sm{S}}$
right-induced from $\spc$ by the adjunction
$$\adjunction{\sPre{\Sm S}}{\caspreS,}{\widetilde{F}}{\widetilde{U}}$$
where $\widetilde{F}$ denotes the free commutative algebra functor and $\widetilde{U}$ the forgetful functor.
\end{lem}

We denote this model category structure by $\caspc$ and call it the \emph{category of motivic commutative algebras}.

\subsection{Equivariant motivic spaces}\label{sec:Gobj}

For any group $G$, let $\Gmot$ denote the category of objects in $\sPre{\Sm{S}}$ equipped with an (objectwise) left $G$-action and of $G$-equivariant morphisms or, equivalently, the category of presheaves of simplicial left $G$-sets on $\Sm{S}$. Similarly, let $\motG$ denote the category of presheaves of simplicial right $G$-sets on $\Sm S$.
There are two adjoint pairs 
\[
\xymatrix@C=5pc {
\Gmot\
\ar@<-0ex>[r]|-(.5){\ U\ }
&
\spre (\Sm{S}),\
\ar@/^1pc/[l]^-(.5){\Hom(G,-)} \ar@/_1pc/[l]_-(.5){G\ti-}
}
\]
with $(G\ti-)\dashv U\dashv \Hom (G,-)$, where $G\times -$ denotes the objectwise product with the constant simplicial set $G$ and similarly for $\Hom (G,-)$.

In the next lemma we define both left- and right-induced model structures on $\Gmot$ from $\spc$, so that, as required in \fullref{conv:modelcat2}, both the adjunction between trivial $G$-action functor and the fixed point functor and the adjunction between the trivial $G$-action functor and the orbit functor are Quillen pairs,  the first with respect to the left-induced structure and the second with respect to the right-induced structure. There are well known conditions under which a right adjoint from a category $\cat{C}$ to a model category $\sM$ creates a \emph{right-induced model structure} on $\cat{C}$ (see, for example, \cite{schwede-shipley}). For a discussion of the dual situation of left-induction, see \fullref{appendix}.

\begin{lem}\label{lem:left_ind_unstable} There are left- and right-induced model structures on $\Gmot$, created by $U$ from the model structure $\spc$.
\end{lem}

We call these structures the \emph{left- and right-$U$-lifted model structures} and denote them by $\Gspcl$ and $\Gspcr$. Similar left- and right-induced model structures, denoted $\spclG$ and $\spcrG$  exist on $\motG$. By the definition of induced model structures, a morphism in $\Gspcl$ is a weak equivalence or cofibration if it is so in $\spc$, while a morphism in $\Gspcr$ is a weak equivalence or fibration if it is so in $\spc$.

\begin{proof}
We first treat the case of the left-induced structure.
We start with the injective model structure on $\spre(\Sm{S})$, where weak equivalences and cofibrations are defined objectwise. Consider the adjoint pair

\[ \adjunction{\Gmot}{\spre (\Sm{S}).}{U}{\Hom(G,-)}
\]

In the global injective model structure on $\spre(\Sm{S})$, all objects are cofibrant. 
Moreover, the usual cylinder object construction in the global injective model structure on $\spre(\Sm{S})$ lifts to provide a cylinder in $\Gmot$. The dual of Quillen's path object argument \cite[Theorem 2.2.1]{HKRS} therefore implies the existence of a model structure on $\Gmot$ left-induced from the global injective model structure.

By \fullref{cor:left_ind_left_Bousf}, it follows that  there is a left-induced model structure on $\Gmot$ created by $U$ from $\spc $, as this is just a left Bousfield localisation of the global injective model structure.

The existence of the right-induced structure on $\Gmot$ follows immediately from \cite[Theorem 4.1]{schwede-shipley}, since $\spc$ is cofibrantly generated and monoidal and satisfies the monoid axiom.
\end{proof}

The next lemma ensures that the final condition of \fullref{conv:modelcat2} holds for motivic spaces.
\begin{lem}\label{lem:homX} For every fibrant object $X$ in $\spc$, the adjunction
$$\adjunction{\Gspcl}{\big((\spcrG)^{op}}{\Hom (-,X)}{\Hom (-,X)}$$
is a Quillen pair.
\end{lem}

\begin{proof}  Let $j\colon Y \to Z$ be a cofibration in $\Gspcl$, i.e., $Uj\colon UY \to UZ$ is a cofibration in $\spc$.  Since $\spc$ is a monoidal model category, and $X$ is fibrant, 
$$\Hom (Uj, X)\colon \Hom (UZ, X) \to \Hom (UY, X)$$ 
is a fibration in $\spc$.  Because $U\Hom (j,X) = \Hom (Uj, X)$, we conclude that 
$$\Hom (j,X)\colon \Hom (Z,X) \to \Hom (Y,X)$$ is a fibration in $\spcrG$ and therefore represents a cofibration in $\big((\spcrG)^{op}$.   A similar argument shows that $\Hom (-,X)$ also preserves acyclic cofibrations.  
\end{proof}

The following pair of lemmas, and their analogues later in this section, which are easy to prove, are needed for \fullref{conv:modelcat2}.
 
\begin{lem}\label{lem:trivOrb} The adjunction between the trivial-right-$G$-action and $G$-orbits functors 
\[
\adjunction{\spcrG}{\spc}{(-)_{G}}{{\trivial_{G}}}
\]
is a Quillen pair.
\end{lem}

\begin{proof} Since $U \trivial_{G} =\Id$, and $U$ creates the model structure $\spcrG$, it is immediate that $\trivial_{G}$ is a right Quillen functor. 
\end{proof}

\begin{lem}\label{lem:trivFP} The adjunction between the trivial-left-$G$-action and $G$-fixed points functors 
\[
\adjunction{\spc}{\Gspcl}{{\trivial_{G}}}{(-)^G}
\]
is a Quillen pair.
\end{lem}

\begin{proof} Since $U \trivial_{G} =\Id$, and $U$ creates the model structure $\Gspcl$, it is immediate that $\trivial_{G}$ is a left Quillen functor. 
\end{proof}

Note that in the previous lemmas we could have  interchanged the roles of $\Gspc$ and $\spcG$, i.e., the left and right $G$-actions, and the associated  adjunctions would still be Quillen pairs.

\begin{defn}  Let $X$ be an object in $\Gmot$.  For any fibrant replacement $X^{f}$ of $X$ in $\Gspcl$, we call $(X^{f})^{G}$ \emph{(a model for) the homotopy fixed points} of $X$. 
\end{defn}

\begin{notation} Abusing notation in the standard manner, any model for the homotopy fixed points of $X$ is denoted $X^{hG}$, suppressing explicit reference to the fibrant replacement.
\end{notation}

As in the case of topological spaces with a $G$-action, one model of the homotopy fixed points of a motivic $G$-space can be constructed by mapping out of a free $G$-space that is non-equivariantly contractible. 
Let $E_{\bullet}G$ denote the simplicial $G$-set that is the usual one-sided bar construction on $G$, which can also be viewed as an objectwise-constant simplicial object in $\Gmot$ or,  after forgetting the $G$-action, in $\spre(\Sm S)$. Its geometric realization in $\spre(\Sm S)$ is contractible.

For any fibrant $Y$ in $\spc$, applying $\Hom (-,Y)$ levelwise to $E_{\bullet}G$ gives rise to a cosimplicial object in $\Gmot$. {It} is levelwise fibrant with respect to the model structure $\Gspcl$, since $\Hom (G\times K, Y)\cong \Hom (G, \Hom (K, Y))$ for any set $K$.  Applying Dugger's cosimplicial replacement construction \cite[Section 5.7]{dugger:hocolim} gives rise to a Reedy fibrant cosimplicial object $\mathsf{crep}\Hom (E_{\bullet}G, Y)$ in $\Gspcl$, the totalization of which is fibrant by {Bousfield} \cite[2.8]{BousfieldCosimplicialResolutions} and {is} a model of the homotopy limit of $\Hom (E_{\bullet}G, Y)$. Note that 
$$\mathsf{crep}\Hom (E_{\bullet}G, Y){\cong} \Hom (\mathsf{srep} E_{\bullet}G, Y),$$
where $\mathsf{srep}$ denotes Dugger's simplicial replacement construction \cite[Section 4.4]{dugger:hocolim}.    

\begin{notation}
We let $EG$ denote $|\mathsf{srep} E_{\bullet}G|$, the geometric realization of $\mathsf{srep} E_{\bullet}G$, which is weakly equivalent to the geometric realization of $E_{\bullet} G$, since $E_{\bullet}G$ is  Reedy cofibrant. For any fibrant $Y$ in $\spc$, 
$$\Hom(EG, Y)=\Hom \big(|\mathsf{srep} E_{\bullet}G|, Y\big)\cong \operatorname{Tot} \mathsf{crep}\Hom (E_{\bullet}G, Y).$$ 
\end{notation}

\begin{lem}\label{lem:unstable_fibrant_repl} Let $X$ be an object in $\Gmot$. If $(UX)^f$ is a fibrant replacement of $UX$ in $\spc$, then ${\Hom}\big(EG,(UX)^f\big)$ is a fibrant replacement of $X$ in $\Gspcl$.
\end{lem}

\begin{proof}  We know already that ${\Hom}\big(EG,(UX)^f\big)$ is fibrant in $\Gspcl$.  
There is a morphism of cosimplicial objects in $\Gmot$ from the constant cosimplicial object on $X$ to   ${\Hom}\big(\mathsf{srep} E_{\bullet}G, UX\big)$, given essentially by iterating the unit map of the $U\dashv \Hom (G,-)$ adjunction. Composing with the morphism 
$$\Hom (\mathsf{srep} E_{\bullet}G, UX)\to {\Hom}\big(\mathsf{srep} E_{\bullet}G,(UX)^f\big)$$ induced by the replacement map $UX \to (UX)^{f}$ and then totalizing gives rise to a map $X\to  {\Hom}\big(EG,(UX)^f\big)$.  
To see that it is a weak equivalence, observe that after applying the functor $U$, which commutes with both limits and colimits, we obtain a morphism  
$$UX \lra U\mathrm{Tot} {\Hom}\big(\mathsf{srep} E_{\bullet}G,(UX)^f\big)=U{\Hom} \big(EG, (UX)^f\big)$$
that  is a weak equivalence, since $UX \to U\mathrm{Tot} \Hom(\mathsf{srep} E_{\bullet}G,UX)$ is a simplicial homotopy equivalence by Bousfield \cite[Propostion 2.13]{BousfieldCosimplicialResolutions}
and $UX \to (UX)^{f}$ is a weak equivalence.
\end{proof}

For the model categories of equivariant motivic spaces to fit into the formal homotopic Galois theory framework {of \fullref{sec:gal-formal}}, we need the following compatibility between its model and monoidal structure.

\begin{prop}\label{prop:monoidal_mod_right_left_ind} With respect to the objectwise product, $\Gspcl$ and $\spcrG$  are cofibrantly generated, monoidal model categories that satisfy the monoid axiom.
\end{prop}

\begin{proof} In the case of the left-induced structure, the result is an immediate consequence of {Hess and Shipley} \cite[Proposition A.9]{hess-shipley:waldhausen}, since the model category structure $\Gspcl$ is left-lifted using a strong monoidal functor. For the right-induced structure, we refer to Schwede and Shipley \cite[Theorem 4.1]{schwede-shipley} for the proof that $\spcrG$ is cofibrantly generated. 
To see that it is a monoidal model category it is enough to verify the pushout-product axiom for generating cofibrations and generating acyclic cofibrations, which are of the form $G\times f$, where $f$ is a generating (acyclic) cofibration in $\spc$. 

Observe that for every pair of  morphisms $f$ and $g$ in $\spc$,
$$(G\times f) \Box (G\times g) \cong (G\times G) \times (f\Box g),$$
where $-\Box -$ denotes the pushout-product. If $f$ and $g$ are cofibrations, then $f\Box g$ is a cofibration, which is acyclic if $f$ or $g$ is acyclic, since $\spc$ is a monoidal model category.  By definition of the right-induced model structure, if follows that  $G\times (f\Box g)$ is a cofibration (respectively, acyclic cofibration) in $\spcrG$. Since the functor $G\times -$ preserves monomorphisms and therefore 
\[U: \spcrG \lra \spc\] 
is also a left Quillen functor, $G\times (f\Box g)$ is a cofibration (respectively, acyclic cofibration) in $\spc$ and therefore $G\times G\times (f\Box g)$ is a cofibration (respectively, acyclic cofibration) in $\spcrG$.

Finally, the right-induced model structure satisfies the monoid axiom because acyclic cofibrations in $\spcrG$ are also acyclic cofibrations in $\spclG$, and the left-induced model structure satisfies the monoid axiom.
\end{proof}

We can now apply {Schwede and Shipley} \cite[Theorem 4.1(3)]{schwede-shipley} to obtain the desired model structures on equivariant motivic algebras.

\begin{cor}\label{cor:Galgright}
There exist model category structures on the category $\aGspreS$ of algebras in $\Gmot$, right-induced from $\Gspcl$, and on the category $\aspreSG$ of algebras in $\motG$ right-induced from $\spcrG$ by the adjunctions
$$\adjunction{\Gmot}{\aGspreS}{F}{U}\text{ and }\adjunction{\motG}{\aspreSG}{F}{U}, $$
where $F$ denotes the free associative algebra functor and $U$ the forgetful functor.
\end{cor}

We denote these model structures by $\aGspcl$ and $\aspcGr$.

\begin{rem}\label{rem:Galgrigh}
Since the model structures $\aspc$ and $\aGspcl$ are right-induced from $\spc$ and $\Gspcl$, respectively, the  $\big(\trivial_{G}, (-)^G\big)$-adjunction lifts to a Quillen pair
\[
\adjunction{\aspc}{\aGspcl,}{\trivial}{(-)^G}
\]
and the fibrant replacement of an object $A$ in $\aGspcl$ is given by ${\Hom}\big(EG,(UA)^{f}\big)$, where $(UA)^{f}$ is the fibrant replacement of $UA$ in  $\aspc$. 

Similarly, the $\big((-)_{G}, \trivial_{G}\big)$-adjunction lifts to a Quillen pair
\[
\adjunction{\aspcGr}{\aspc.}{(-)_{G}}{\trivial_{G}}
\]
\end{rem}

The category of commutative algebras in  equivariant motivic spaces also admits lifted model structures.

\begin{prop}\label{prop:Gcommalgright}
There exist model category structures on the category $\caGspreS$ of commutative algebras in $\Gmot$, right-induced from $\Gspcl$, and on the category $\caspreSG$ of commutative algebras in $\motG$ right-induced from $\spcrG$ by the adjunctions
$$\adjunction{\Gmot}{\caGspreS}{\widetilde F}{\widetilde U}\text{ and }\adjunction{\motG}{\caspreSG}{\widetilde F}{\widetilde U}, $$
where $\widetilde F$ denotes the free commutative algebra functor and $\widetilde U$ the forgetful functor.
\end{prop}

{We denote these model structures by $\caGspcl$ and $\caspcGr$. }

\begin{proof}  In the case of the model structure  $\caGspcl$ on $\caGspreS$ right-induced from $\Gspcl$, we
apply \fullref{thm:square} to the diagram

\[ \xymatrix@R=4pc@C=4pc{ \sPre {\Sm S} \ar@{}[r]|-{\perp} \ar@<-1ex>[d]_-{\widetilde F} \ar@<-1ex>[r]_-{\Hom (G,-)} & G\text{-}\sPre {\Sm S}  \ar@<1ex>[d]^-{\widetilde F}\ar@<-1ex>[l]_-U\\  
\ar@{}[r]|-{\top} \ar@{}[u]|-{\dashv} \caspreS\ar@<1ex>[r]^-{\Hom (G,-)} \ar@<-1ex>[u]_-{\widetilde U} & \caGspreS, \ar@<1ex>[u]^-{\widetilde U} \ar@<1ex>[l]^-U\ar@{}[u]|-{\vdash} }\] 
where $\sPre {\Sm S}$ is equipped with its motivic injective model structure $\spc$, $G\text{-}\sPre {\Sm S}$ with the left-induced structure $\Gspcl$, and $\caspreS$ 
with the right-induced structure of {\fullref{cor:calgright}}, $\caspc$.

To establish the existence of the model structure on $\caspreSG$ right-induced from $\spcrG$, we observe first that  \fullref{thm:square} applies equally well to the case of right $G$-actions and implies moreover that there also exists a model structure on $\caspreSG$ left-induced by $U\colon \caspreSG\to \caspreS$ from {$\caspc$}. Since $\caspreSG$ is isomorphic to  $G/\caspreS$, a full subcategory of the under-category $G/\aspreS$, with respect to which $U$ can be seen as the functor forgetting the map from $G$,  the model structure left-induced by $U$ is necessarily the same as the model structure right-induced by $U$ (i.e., $U$ creates all three classes of maps: cofibrations, fibrations and weak equivalences). Since $U$ creates the model structure $\spcrG$, and $U\widetilde{U}=\widetilde{U}U$, it follows that $\widetilde{U}$ creates a model structure on {$\caspreSG$}, right-lifted from $\spcrG$.
\end{proof}

As before, easy proofs analogous to those of \fullref{lem:trivFP} and \fullref{lem:trivOrb} establish the following results.

\begin{lem} The adjunction
{\[
\adjunction{\caspc}{\caGspcl}{\triv G}{(-)^G}
\] }
 is a Quillen pair.
\end{lem}

\begin{lem} The adjunction
{\[
\adjunction{\caspcGr}{\caspc}{(-)_{G}}{\triv G}
\] }
is a Quillen pair.
\end{lem}

As above, we obtain fibrant replacements of the desired special form for commutative, motivic $G$-algebras.

\begin{prop} \label{prop:fibreplacement_pos}    
For any commutative, motivic $G$-algebra $Y$, the commutative algebra ${\Hom}\big(EG,(UY)^f\big)$ is a fibrant replacement in $\caGspcl$, where $(UY)^f$ is  a fibrant replacement of $UY$ in {$\caspc$}.  
\end{prop}

\begin{proof} The desired result follows by the same argument as in \fullref{lem:unstable_fibrant_repl}.
\end{proof}

\begin{rem} One can prove results in the pointed case analogous to all those in this section, starting from the $\pspc $ instead of  $\spc$, {replacing the functors $\Hom(G,-)$ with $\Hom(G_+,-)$ and $G \times - $ with $G_+\smsh -$.}
\end{rem}

\subsection{Motivic spectra}\label{sec:mot-spectra}

For any object $X$ in $\spre_*(\Sm{S})$, let 
$$\spe \big(\spre_*(\Sm{S}),X\big)$$ 
denote the category of \emph{symmetric $X$-spectra} of objects in $\spre_*(\Sm{S})$ {(see Hovey \cite{hovey-spectra})}.  Objects in $\spe \big(\spre_*(\Sm{S}),X\big)$  are sequences  $(Y_{0},Y_{1},..., Y_{n},...)$ of objects in  $\spre_*(\Sm{S})$, where $Y_{n}$ is equipped with a $\Sigma_{n}$-action for every $n$, together with structure maps $Y_{n}\wedge X \to Y_{n+1}$ that are appropriately equivariant.  The category of symmetric $X$-spectra is symmetric monoidal, with respect to the graded smash product $\wedge$ over a symmetric sequence built from $X$ (given by the usual coequaliser in symmetric sequences).

There is a model structure on $\spe \big(\spre_*(\Sm{S}),X\big)$ in which both fibrations and weak equivalences are lifted levelwise from {$\pspc$}, 
by \cite[Theorem 8.2]{hovey-spectra}, which we call the \emph{levelwise {motivic} model structure}.

\begin{defn}\label{defn:mot_spec_mod}
The \emph{stable motivic model structure} on the category of symmetric $X$-spectra, denoted $\spt$, 
is the left Bousfield localization of the levelwise {motivic} model structure with respect to a set of maps described explicitly in \cite[Definition 8.7]{hovey-spectra}. 
\end{defn}

The model category $\spt$ has the same cofibrations as the levelwise {motivic} model structure. We call its weak equivalences the \emph{stable motivic weak equivalences.}
It is compatible with the graded smash product, in the sense that $(\spt, \wedge )$ is a symmetric monoidal model category \cite[Theorem 8.11]{hovey-spectra}.  {Smashing with $X$ levelwise} is a Quillen equivalence on $\spt$.

\begin{lem}\label{lem:monoid_spect}
The monoid axiom is satisfied in  $(\spt, \wedge)$. 
\end{lem}

\begin{proof}
Let $J$ be a class of acyclic cofibrations in $\spt$, and let $Y$ be  a symmetric $X$-spectrum. By Jardine \cite[Proposition 4.19]{JardineMotivicSymSpec}, the class $J \wedge Y$ is comprised of levelwise monomorphisms that are stable weak equivalences.
It is therefore enough to show that the class of maps that are stable equivalences and levelwise monomorphisms is stable under pushouts and transfinite compositions, which is established in the proof of {Hoyois} \cite[Lemma 4.2]{HoyoisFromAlgebraicCobordism}.\end{proof}

\begin{rem}\label{rem:algrighspec}
It follows from \fullref{lem:monoid_spect} that 
the stable motivic model structure $\spt$ right-induces a model structure on its associated category of algebras, giving rise to a model category we denote by $\aspt$.
\end{rem}

\subsection{Equivariant motivic spectra}

We now introduce an equivariant version of the symmetric $X$-spectra considered above, but only in the case when $X$ 
is a motivic space with a trivial $G$-action. This is all we need for the examples treated in \fullref{sec:examples}.

For any finite group $G$ and any object $X$ in $\spre_*(\Sm{S})$, let $$\spe\big(\pGmot, \triv{G}X\big)$$ denote the category of \emph{symmetric $\triv{G}X$-spectra} of objects in $\pGmot$, defined analogously to $\spe \big(\spre_*(\Sm{S}),X\big)$.  Note that there is an obvious isomorphism of categories 
$$\spe \big(\pGmot, \triv{G}X\big)\cong G\text{-}\spe \big(\spre_*(\Sm{S}), X\big).$$

Since $U: \pGmot \lra \spre_*(\Sm{S})$ is a strong symmetric monoidal functor, the functors $U$ and $\Hom (G_{+},-)$  lift to levelwise adjoint functors on spectra, 
as does levelwise smashing with $G_{+}$,
giving rise to two adjoint pairs 
\begin{equation}\label{eqn:Gsmsh}
\xymatrix@C=5pc {
\spe \big(\pGmot, \triv{G}X\big)\
\ar@<-0ex>[r]|-(.4){\ U\ }
&
\spe (\spre_*(\Sm{S}), X),\
\ar@/^1.5pc/[l]^-{\Hom(G_{+},-)} \ar@/_1.5pc/[l]_-{G_{+}\wedge-}
}
\end{equation}
with $(G_{+}\wedge-)\dashv U\dashv \Hom (G_{+},-)$.

As in the unstable case, in the next lemma we define both left- and right-induced model structures on $\spe \big(\pGmot, \triv{G}X\big)$ from $\spt$, so that, as required in \fullref{conv:modelcat2}, both the adjunction between the trivial $G$-action functor and the fixed point functor and the adjunction between the trivial $G$-action functor and the orbit functor are Quillen pairs,  the first with respect to the left-induced structure and the second with respect to right-induced structure.

\begin{prop}\label{U-lifted-stable}The category $\spe \big(\pGmot, \triv{G}X\big)$ admits  cofibrantly generated model structures left- and right-induced by $U$ from $\spt$ through the adjunctions {of (\ref{eqn:Gsmsh})}.
\end{prop}

We denote these structures 
$$\Gsptl \quad\text{and}\quad \Gsptr.$$
Similar left- and right-induced model structures exist on $\spe \big(\pmotG, \triv{G}X\big)$.

\begin{proof}
We start with the left-induced structure. As in the unstable case, there is a model structure on $\spe\big(\pGmot, \triv{G}X\big)$ in which both fibrations and  weak equivalences are lifted levelwise from  $\pGspcl$, which we call the \emph{levelwise $U$-lifted model structure}. {Furthermore}, $U\dashv\Hom(G_+,-)$ is a Quillen pair with respect to the levelwise $U$-lifted model structure on the category $\spe \big(\pGmot, \triv{G}X\big)$ and the {levelwise motivic model structure on $\spe (\spre_*(\Sm{S}), X\big)$}.

Since $U$ is a left Quillen functor and preserves all weak equivalences, \fullref{prop:left_ind_along_leftQuillen} implies that there exists a model structure on $\spe \big(\pGmot, \triv{G}X\big)$ that is left-induced by $U$ from the levelwise motivic model structure on  $\spe (\spre_*(\Sm{S}), X\big)$.  It then follows by \fullref{cor:left_ind_left_Bousf}  that the functor $U$ left-induces a model structure on $\spe \big(\pGmot, \triv{G}X\big)$  from $\spt$. {Since $\spt$ is cofibrantly generated, so is this new model category.}

Just as in the case of motivic spaces, the existence of the right-induced structure follows immediately from \cite[Theorem 4.1]{schwede-shipley}.
\end{proof}

The next lemma, the proof of which is essentially identical to that of \fullref{lem:homX}, ensures that the final condition of \fullref{conv:modelcat2} holds for motivic spectra.

\begin{lem}\label{lem:homXstable} For every fibrant object $X$ in $\spt$, the adjunction
$$\adjunction{\Gsptl}{\big(\sptrG)^{op}}{\Hom (-,X)}{\Hom (-,X)}$$
is a Quillen pair.
\end{lem}

As in the unstable case, we can deduce  the existence of fibrant replacements in $\Gsptl$ of a particular form.

\begin{lem}\label{stable_fib_rep} For any symmetric {$\triv GX$}-spectrum $Y$, 
$\Hom(EG,(UY)^f)$ is a fibrant replacement of $Y$ in  $\Gsptl$, where $(UY)^f$ is a fibrant replacement of $UY$ in $\spt$.
\end{lem}

\begin{proof} Since $(UY)^f$ is in particular levelwise fibrant (the fibrations in the stable motivic model structure are a subclass of those in the levelwise {motivic} model structure), the natural map  
$$Y \lra \Hom(EG,(UY)^f)$$  is a levelwise weak equivalence by \fullref{lem:unstable_fibrant_repl}, and thus a stable weak equivalence. 
An argument analogous to that in the proof of \fullref{lem:unstable_fibrant_repl} implies that $\Hom(EG,(UY)^f)$ is fibrant in $\Gsptl$.
\end{proof} 

Easy proofs, essentially identical to those of \fullref{lem:trivOrb} and \fullref{lem:trivFP}, establish the following results.

\begin{lem}\label{lem:stabletrivOrb} The adjunction between the trivial-right-$G$-action and $G$-orbits functors 
\[
\adjunction{\sptrG}{\spt}{(-)_{G}}{{\trivial_{G}}}
\]
is a Quillen pair.
\end{lem}

\begin{lem}\label{lem:stabletrivial_fixedpoints}
The $\big(\triv G, (-)^{G}\big)$-adjunction
$$\adjunction{\spt}{\Gsptl}{\triv G}{(-)^G}$$
is a Quillen pair.
\end{lem}

In order to ensure that equivariant motivic spectra fit into the formal framework for homotopical Galois theory, we need {the following result.}

\begin{lem} {If $\wedge$ denotes the usual smash product of symmetric $X$-spectra, then 
$$\big(\Gsptl, \wedge\big)\quad\text{and}\quad\big(\sptrG,\wedge\big)$$ 
are monoidal model categories that satisfy the monoid axiom.}
\end{lem}

Note that a smash product of $G$-objects is endowed with the diagonal $G$-action.

\begin{proof} Since $U$ is a strong monoidal functor, and $\Big(\spt, \wedge \Big)$ is a monoidal model category satisfying the monoid axiom, the case of the left-induced structure is an immediate consequence of {Hess and Shipley} \cite[Proposition A.9]{hess-shipley:waldhausen}.
The proof for the right-induced model structure follows the same pattern as the proof of \fullref{prop:monoidal_mod_right_left_ind}.
\end{proof}

We can now apply {Schwede and Shipley} \cite[Theorem 4.1(3)]{schwede-shipley} to obtain the desired model structures on equivariant motivic ring spectra. 

\begin{cor}\label{cor:Galg_spectra}
There exist  model structures on the categories $\aGspt$ and $\asptG$  of algebras in $\spe \big(\pGmot, \triv{G}X\big)$ and  $\spe \big(\pmotG, \triv{G}X\big)$, respectively, right-induced 
{from $\Gsptl$ and $\sptrG$}
by the adjunction
between the free associative algebra functor $F$ and the forgetful functor.
\end{cor}

We denote these model structures by $\aGsptl$ and $\asptrG$.

\begin{rem}\label{rem:homeg_fib}
Since the model structures on $\aspt$ and $\aGsptl$ are right-induced from $\spt$ and $\Gsptl$ respectively, the lifted adjunction 
\[
\adjunction{\aspt}{\aGsptl}{\triv G}{(-)^G}
\]
is still a Quillen pair by \fullref{lem:stabletrivial_fixedpoints}.  The analogous result holds for the 
{$(-)_{G}\dashv \triv G$} adjunction as well. Moreover, the fibrant replacement of an algebra $A$ in $\aGsptl$ is given by ${\Hom}\big(EG,(UA)^{f}\big)$, where $ (UA)^{f}$ denotes the fibrant replacement of $UA$ in $\aspt$. 
\end{rem}

\subsection{Commutative algebras in (equivariant) motivic spectra}\label{section:comalg_motSp}
Let $G$ be a finite group, and let $X$ be any object in $\spre_*(\Sm{S})$.  Let $X\otimes -$ denote the left adjoint to the forgetful functor from $\spe \big(\spre_*(\Sm{S}),X\big)$ to the category of symmetric sequences in $\spre_*(\Sm{S})$.

Let $\caspt$ denote the category of commutative algebras in $\spe \big(\spre_*(\Sm{S}),X\big)$, and let
\begin{equation}\label{eqn:freecomm}\adjunction{\spe \big(\spre_*(\Sm{S}),X\big)}{\caspt}{\widetilde F}{\widetilde{U}}\end{equation}
denote the free commutative algebra adjunction.

\begin{prop}[{Hornbostel} {\cite[Theorems 3.4 and 3.6]{Hornbostel}}]\label{prop:hornbostel}The category $\spe \big(\spre_*(\Sm{S}),X\big)$ admits a model category structure in which the weak equivalences are the stable motivic weak equivalences, and the class of cofibrations is   $X \otimes M$-cof, where $M$ is a class of monomorphisms in symmetric sequences (and cof is explained in \fullref{notation:cof}). This model category structure lifts to a model category stucture on $\calg\spe$, right-induced by the adjunction {(\ref{eqn:freecomm})}.
\end{prop}

We call the model structure of the proposition above the \emph{positive flat stable model structure} and denote it $\Pspt$.  We denote the right-induced model structure on the category of commutative motivic algebras by $\Pcaspt$.

We now establish the existence of two types of lifted model structures on categories of commutative equivariant motivic algebras.

\begin{lem}\label{lem:Gmot_spectra_pos_right}
The forgetful functor, right adjoint to $-\wedge G_{+}$,
$$U\colon \spe \big(\pmotG, \triv{G}X\big) \to \spe \big(\spre_*(\Sm{S}),X\big)$$ 
right-induces a model structure on $ \spe \big(\pmotG, \triv{G}X\big)$ from $\Pspt$.
\end{lem}

 We will call this model structure the \emph{right-lifted positive model structure} and denote it $\GPsptr$.
 
\begin{proof} Let  $I=(S^{0},X, X\wedge X, ....)$, the unit $X$-spectrum. Since $\spe \big(\pmotG, \triv{G}X\big)$ is the category of right $I[G_{+}]$-modules in $\spe \big(\spre_*(\Sm{S}),X\big)$, it is enough to show that $\Pspt$ satisfies the conditions of \cite[Theorem 4.1(1)]{schwede-shipley}. Hornbostel showed in \cite[Lemma 3.5]{Hornbostel} that $\Pspt$ is a monoidal model category.
By \cite[Remark 4.2]{schwede-shipley}, it is enough then to show that if $f$ is an acyclic cofibration in $\Pspt$, then $(f\wedge G_+)$-cof$_\mathrm{reg}$ (i.e., the class obtained from $f\wedge G_+$  by taking pushouts and transfinite compositions) consists of acyclic cofibrations.
Since $f\wedge G_+ \cong (\coprod_{|G|}f)_+$, if $f$ is an acyclic cofibration, then so is  $f\wedge G_+$. Since acyclic cofibrations are closed under pushouts and transfinite compositions, we conclude that the desired right-induced model structure on  $ \spe \big(\pmotG, \triv{G}X\big)$ exists.
\end{proof}

\begin{lem}\label{lem:Gmot_spectra_pos}
The forgetful functor, left adjoint to $\Hom(G_+,-)$,
$$U\colon \spe \big(\pGmot, \triv{G}X\big) \to \spe \big(\spre_*(\Sm{S}),X\big)$$ 
left-induces a model structure on $ \spe \big(\pGmot, \triv{G}X\big)$ from $\Pspt$. 
\end{lem}

We call the model structure of the lemma above the \emph{left-lifted positive model structure} and denote it $\GPsptl$.

\begin{proof} Since $\pGspcl$ satisfies the conditions of {Gorchinskiy and Guletskii} \cite[Proposition 1]{Gorchinskiy-Guletskii}, there is a levelwise positive flat model structure on $\spe \big(\pGmot, \triv{G}X\big)$ in the sense of \cite{Gorchinskiy-Guletskii}, based on { $\Gspcl$} .  Hornbostel \cite{Hornbostel} defines what he also calls a levelwise positive flat model structure on $\spe \big(\spre_*(\Sm{S}),X\big)$, from which he derives the positive flat stable model structure of \fullref{prop:hornbostel} by left Bousfield localization. With respect to these two levelwise positive flat  model structures, the forgetful functor $U$ is a left Quillen functor, as it sends generating cofibrations to cofibrations and preserves all weak equivalences. \fullref{prop:left_ind_along_leftQuillen} implies that  $U$ therefore creates  a model structure on $\spe \big(\pGmot, \triv{G}X\big)$,  left-induced from the levelwise positive flat model structure.  It follows then from \fullref{cor:left_ind_left_Bousf}   {that} there is a left-induced model structure on $\spe \big(\pGmot, \triv{G}X\big)$ created by $U$ from $\Pspt$. 
\end{proof}

Let $\caGspt$ denote the category of commutative algebras in the symmetric monoidal category $\spe \big(\pGmot, \triv{G}X\big)$, and let
\begin{equation}\label{eqn:gfreecomm}\adjunction{\spe \big(\pGmot, \triv{G}X\big)}{\caGspt}{\widetilde F}{\widetilde{U}}\end{equation}
again denote the free commutative algebra adjunction, of which there is also a version for right $G$-actions.

\begin{prop} \label{prop:CalgG_model}
The model structure $\GPsptl$ lifts to a model structure on $\caGspt$, right-induced by the adjunction {(\ref{eqn:gfreecomm})}.
\end{prop}

We denote this model category structure by $\caGPsptl$.

\begin{proof}
Since  $U\colon  \spe \big(\pGmot, \triv{G}X\big) \lra \spe \big(\spre_*(\Sm{S}),X\big)$ is strong symmetric monoidal, there is a square of adjunctions
\[ \xymatrix@R=4pc@C=4pc{ \spe \big(\spre_*(\Sm{S}),X\big) \ar@{}[r]|-{\perp} \ar@<-1ex>[d]_-{\widetilde F} \ar@<-1ex>[r]_-{\Hom(G_+,-)} & \spe \big(\pGmot, \triv{G}X\big)  \ar@<1ex>[d]^-{\widetilde F} \ar@<-1ex>[l]_-U \\  
\ar@{}[r]|-{\top} \ar@{}[u]|-{\dashv} \caspt \ar@<1ex>[r]^-{\Hom(G_+,-)} \ar@<-1ex>[u]_-{\widetilde{U}} & \caGspt \ar@<1ex>[u]^-{\widetilde{U} }\ar@<1ex>[l]^-U\ar@{}[u]|-{\vdash} }\] 
in which $U\widetilde F=\widetilde FU$ and $ \widetilde U U= U\widetilde U$.
It follows from  \fullref{thm:square} that the required right-induced model structure exists, when we consider the model structures $\GPsptl$, $\Pspt$, and $\Pcaspt$ in the diagram above.
\end{proof}

\begin{rem}\label{rem:rightliftedGcomalg}  The proposition above holds for right $G$-actions as well.   \fullref{thm:square} implies moreover that there also exists a left-induced model structure on $\casptG$ created by $U\colon \casptG\to \caspt$. Since $\casptG$ is isomorphic to  $I[G_+]/\caspt$, a full subcategory of the under category $I[G_+]/\aspt$, with respect to which $U$ can be seen as the functor forgetting the map from $I[G_{+}]$,  the model structure left-induced by $U$ is necessarily the same as the model structure right-induced by $U$ (i.e., $U$ creates all three classes of maps: cofibrations, fibrations and weak equivalences).
\end{rem}

\begin{prop}\label{prop:GCalg_model}
The model structure $\PsptrG$ lifts to a model category stucture on $\caPsptG$, right-induced by the free commutative algebra adjunction.
\end{prop}

We denote this model structure $\caPsptrG$.

\begin{proof} By \fullref{rem:rightliftedGcomalg} the composite functor $U\widetilde U$ creates a model structure on $\casptG$, right-lifted from $\Pspt$. Since $U$ creates the model structure $\PsptrG$, and $U\widetilde{U}= \widetilde{U}U$, it follows that $\widetilde{U}$ creates a model structure on $\casptG$, right-lifted from $\PsptrG$.
\end{proof}

Once more, easy proofs along the lines of those of \fullref{lem:trivFP} and \fullref{lem:trivOrb} establish the following results.

\begin{lem} The adjunction
\[
\adjunction{\Pcaspt}{\caGPsptl}{\triv G}{(-)^G}
\] 
 is a Quillen pair  and therefore induces a Quillen pair on the level of commutative algebras as well.
\end{lem}

\begin{lem} The adjunction
\[
\adjunction{\caPsptrG}{\Pcaspt}{(-)_{G}}{\triv G}
\] 
 is a Quillen pair  and therefore induces a Quillen pair on the level of commutative algebras as well.
\end{lem}

As above, we also obtain fibrant replacements of the desired special form for motivic $G$-spectra and their commutative algebras.

\begin{prop} \label{prop:fibreplacement_pos}    
For any symmetric $\triv G X$-spectrum $Y$, ${\Hom}\big(EG,(UY)^f\big)$ is a fibrant replacement in  $\GPsptl$, where $(UY)^f$ is  a fibrant replacement of $UY$ in $\Pspt$.  If $A$ is a commutative algebra in $\spe \big(\spre_*(\Sm{S}),\triv GX\big)$, then  ${\Hom}\big(EG_{+},(UA)^f\big)$ is a fibrant replacement in $\caGPsptl$.
\end{prop}

\begin{proof} The desired result follows by the same argument as in \fullref{lem:unstable_fibrant_repl} and \fullref{stable_fib_rep}.
\end{proof}

\subsection{Relation to the formal framework}\label{sec:motivic-framework} 
We conclude this section by relating the symmetric monoidal model categories defined above to the abstract framework for homotopical Galois theory in a symmetric monoidal model category $(\cat M, \otimes, I)$ described in \fullref{sec:gal-formal}.  In particular we show that all of the theory developed in \fullref{sec:gal-formal} holds in the motivic context: homotopical Galois extensions of (pointed or unpointed) commutative motivic monoids and of commutative motivic ring spectra are well defined and satisfy invariance under cobase change, as well as the forward part of a potential Galois correspondence. It remains an open, interesting, and certainly hard question to determine under what conditions one can establish a backward Galois correspondence, analogous to \cite[Theorem 11.2.2]{rognes}.  Inspired by Rognes's proof in the classical case, we hypothesize that  a motivic version of Goerss-Hopkins obstruction theory should suffice to prove a motivic backward Galois correspondence.

\subsubsection*{Motivic spaces}
We consider first $\cat{M} = \spc$, as defined in \fullref{lem:cat_mot}. As mentioned above, the monoidal product  is given by the objectwise product of simplicial sets. The unit $I$ is  the constant simplicial set on one point $\ast$. In this context, the category $\alg$ of algebras in $\cat M$ is given the model structure $\aspc$ defined in \fullref{cor:algright}.  Since $\spc$ is cofibrantly generated and satisfies the monoid axiom, it follows from \cite[Theorem 4.1]{schwede-shipley} that for any motivic algebra $A$, the free $A$-module adjunction
$$\adjunction{\sPre{\Sm{S}}}{\cat {Mod}_{A}(S)}{-\times A}{U}$$
right-induces a model structure on the category $\cat{Mod}_A(S)$ of $A$-modules in $\sPre{\Sm{S}}$, in which $A$ itself is cofibrant, since $\ast$ is cofibrant in $\spc$.  

For $G$ a finite group, one can choose $H$ to be the constant simplicial set on the underlying set of $G$. We then have 
$${}_{H}\cat {Mod}=\Gspcl \quad\text{and}\quad \cat {Mod}_{H}=\spcrG$$
as defined in \fullref{lem:left_ind_unstable}
and 
$${}_{H}\alg =\aGspcl \quad\text{and}\quad \alg_{H} =\aspcGr$$ as in \fullref{cor:Galgright}. Like in the nonequivariant case, since $\Gspc$ is cofibrantly generated and satisfies the monoid axiom, for any $G$-motivic algebra $A$, the free $A$-module adjunction
$$\adjunction{G\text{-}\sPre{\Sm{S}}}{G\text{-}\cat {Mod}_{A}(S)}{-\times A}{U}$$
right-induces a model structure on the category $G\text{-}\cat{Mod}_A(S)$ of $A$-modules in $G\text{-}\sPre{\Sm{S}}$.

Moreover, there are appropriate model structures on categories of commutative algebras in $\sPre{\Sm{S}}$, $G\text{-}\sPre{\Sm{S}}$, and $\sPre{\Sm{S}}\text{-}G$, as established in \fullref{prop:Gcommalgright}.

All of the adjunctions  in \fullref{conv:modelcat2} are indeed Quillen pairs with respect to the model category structures specified above, as either established in detail in \fullref{sec:Gobj} or easily verified from the definitions of these structures.  Moreover, every equivariant motivic space (respectively, algebra or commutative algebra) $X$ admits a fibrant replacement of the form $\Hom \big(EG, (UX)^{f}\big)$, where the fibrant replacement of $UX$ is computed in the underlying category of nonequivariant spaces (respectively, algebras or commutative algebras), whence $\Hom \big(EG, (UX)^{f}\big)^{G}$ is a model of the homotopy fixed points of $X$.

\subsubsection*{Pointed motivic spaces}
When $\cat{M} = \pspc$ as described in \fullref{lem:cat_mot}, the situation is analogous to that above. In this case, the monoidal  product is given by the levelwise smash of simplicial sets, $I$ is the constant simplicial set on two points $\ast_{+}$. 
For $G$ a finite group, $H$ is the constant simplicial set $G_+$. 
As in the unpointed case, all of the adjunctions  in \fullref{conv:modelcat2} are indeed Quillen pairs with respect to the model category structures specified here.

\subsubsection*{Motivic spectra}  

When $\cat{M} =  \spt$ as defined in \fullref{defn:mot_spec_mod}, the unit $I$ is the $X$-spectrum $(S^{0},X, X\wedge X, ....)$. In this context, the category $\alg$ of algebras in $\cat M$, also known as \emph{motivic ring spectra}, is the model category $\aspt$ of \fullref{rem:algrighspec}. Since $\spt$ is cofibrantly generated and satisfies the monoid axiom, it follows from \cite[Theorem 4.1]{schwede-shipley} that for any motivic ring spectrum $A$, the free $A$-module adjunction
$$\adjunction{\spe \big(\spre_*(\Sm{S}),X\big)}{\cat {Mod}_A\spt}{-\wedge A}{U}$$
right-induces a model structure on the category $\cat{Mod}_A\spt$ of $A$-modules in $\spe \big(\spre_*(\Sm{S}),X\big)$, in which $A$ itself is cofibrant, since $I$ is cofibrant in $\spt$.  

For $G$ a finite group, the Hopf algebra $H$ is $I[G]=\bigvee_{G}I$, and
 $${}_{H}\cat{Mod}=\Gsptl\quad\text{and}\quad \cat{Mod}_{H}=\sptrG$$ 
 of \fullref{U-lifted-stable}.  Moreover
 $${}_{H}\cat{Alg}=\aGsptl\quad\text{and}\quad \cat{Alg}_{H}=\asptrG$$
 of \fullref{cor:Galg_spectra}. By \fullref{stable_fib_rep} and \fullref{rem:homeg_fib}, we again obtain a special model for homotopy fixed points of any equivariant motivic (ring) spectrum, where $EG_{+}$ replaces $EG$.

The category $\caspt$ of commutative algebras in $\spe \big(\spre_*(\Sm{S}),X\big)$ admits the model structure  of \fullref{prop:hornbostel}, while
$${}_{H}\caspt= \caGPsptl \quad\text{and}\quad \caspt_{H}=\caPsptrG,$$
by  \fullref{prop:CalgG_model} and \fullref{prop:GCalg_model}.
Moreover \fullref{prop:fibreplacement_pos} implies the usual formula for the homotopy orbits also holds for commutative, equivariant motivic ring spectra.

Each of the adjunctions  in \fullref{conv:modelcat2} is indeed a Quillen pair with respect to the model category structures specified above, as has either already been verified above or can easily be seen from the definitions of the model structures.

\section{Examples of motivic Galois extensions}\label{sec:examples}

In \fullref{sec:em}, we study a motivic analogue of Galois extensions of Eilenberg--MacLane spectra $HR \to HT$ (see Rognes \cite[Proposition 4.2.1]{rognes}).  
In \fullref{sec:kgl}, we consider the motivic analogue of the classical $C_2$--Galois extension $KO \ra KU$. 
We will show that this is a motivic Galois extension under certain conditions on our base scheme $S$.

\subsection{Preliminaries}
Throughout this section $\SH S$ denotes the homotopy category of $\mathsf{Sp}_{\mathbb P^1}(S)$. We let $\mathbb S$ denote the motivic sphere spectrum.
For any finite group $G$, we specialize the general definition of Galois extension from \fullref{defn:galois-ext} to the framework of motivic 
$\trivial_{G}\mathbb P^1$-spectra, 
as set up in \fullref{sec:motivic-framework}. The Hopf algebra we consider is the suspension spectrum of $G_{+}$, which we also denote $G_+$, in the standard abuse of notation.
Let $\vp:  \triv{G_{+}}{A}\to B^{\circlearrowleft G_{+}}$ be commutative Galois data in $G\text{-}\mathsf{Sp}_{\mathbb P^1}(S)$.

Henceforth, we  implicitly work in the homotopy category. We will use $\smsh$ to denote $\otimes^{\mathbb{L}}$ and $F(-,-)$ to denote $\mathbb{R}\Hom(-,-)$.
Further, we use $B^{hG}$ and $\vp ^{hG}$ to denote what should really be called $B^{hG_{+}}$ and $\vp ^{hG_{+}}$. In this spirit, the definition of a Galois extension of motivic spectra can be formulated as follows.

\begin{defn}
Commutative Galois data $\vp:  \triv{G_{+}}{A}\to B^{\circlearrowleft G_{+}}$ is a \emph{homotopical Galois extension} if the following two conditions hold in $\SH S$.
\begin{enumerate}
\item The map $\beta_{\vp}\colon B\smsh_A B \to F(G_+, B)$ is an isomorphism.
\smallskip
 \item The map $\vp ^{hG}\colon A \to B^{hG}$ is an isomorphism.
\end{enumerate}
\end{defn}
Since $G$ is finite and discrete,  $F(G_+, B) \cong \prod_G B$ in $\SH S$, and the map
\begin{align}\label{eqn:mapFGB}
{\beta_{\vp}\colon}B\smsh_A B \to F(G_+, B) \cong \prod_G B
\end{align}
is simply the map whose factors are the composites $B\smsh_A B  \xrightarrow{1 \smsh_A g} B\smsh_A B \xra{\bar\mu} B$.

Finally, we let $\underline{\pi}_{p,q}B$ denote the presheaf of bigraded abelian groups on $\Sm S$
\begin{equation}\label{eqn:presheaf}
\underline{\pi}_{p,q}B(U) = \Hom_{\SH{S}}(\mmS^{p,q}\smsh\Sigma^{\infty}U_{+}, B).
\end{equation}
The presheaf $\spi_{p,q}(-)$ detects weak equivalences of motivic spectra (see Jardine \cite[Lemma 3.7]{JardineMotivicSymSpec}) and therefore detects isomorphisms in the homotopy category $\SH{S}$. For $S$ our base scheme, we let
\[ \pi_{p,q}B := \spi_{p,q}B(S). \]
The functor $\pi_{p,q}(-)$ detects isomorphisms in $\SH{S}$ between \emph{cellular} motivic spectra, that is, those spectra that are obtained from the stable motivic spheres $\mmS^{p,q}$ by iterated homotopy colimits (see Dugger and Isaksen \cite[Corollary 7.2]{dugisa}). Note that the $G_{+}$--action on $B$, which gives rise to a map of $A$--algebras $g:B\to B$ for every $g\in G$,  induces a  $G$-action on $\spi_{*,*}B$.

\begin{warn}
Some authors use $\underline{\pi}_{*,*}$ to denote the associated Nisnevich sheaf.
\end{warn}

\subsection{Motivic Eilenberg-MacLane spectra}\label{sec:em}
We let $S$ be a scheme satisfying the conditions of \cite{hoyois}, i.e., $S$ is Noetherian of finite Krull dimension. Let $A$ be an abelian group and $\HA$  the motivic Eilenberg-MacLane spectrum representing motivic cohomology, so that for any smooth scheme $X$,
\[ H^{p,q}(X,A) \cong \Hom_{\SH S}(\Sigma^{\infty}X_{+}, \mmS^{p,q}\smsh \HA) = \spi_{-p,-q} \HA (X).\]
An explicit construction of the motivic spectrum $\HA$ is described in Hoyois \cite[Section 4.2]{hoyois}. If $A$ is a $G$--module,
the functoriality of the construction induces an action of $G$ on $\HA$, and $\HA$ is a $G$-object in motivic spectra.

The main result of this section is the following theorem.
\begin{thm}\label{thm:galEM}
Let $G$ be a finite group, and let $R \to T$ be a homomorphism of commutative rings with $G$ acting on $T$ via $R$--linear maps. Then $R \to T$ is a Galois extension of rings if and only if $\HR \to \HT$ is a {homotopical} $G$--Galois extension of motivic ring spectra.
\end{thm}

To prove \fullref{thm:galEM}, we will use two spectral sequences. The first is a homotopy fixed point spectral sequence. Its construction uses the cellular filtration of $EG$ and is analogous to the classical construction (see \fullref{sec:Gobj} for discussion on $EG$ and \cite[Chapter XI, Section 7]{bouskan} for discussion on the spectral sequence). This spectral sequence already appears in the literature in Isaksen and Shkembi \cite[Theorem 3.8]{isashk} and Berrick et al. \cite[(1-d)]{bako}.

\begin{prop}\label{prop:hfss}
Let $G$ be a finite group and $X$ be a $G$--spectrum. There is a tri-graded spectral sequence 
\[E_2^{s,(p,q)}=H^s(G, \spi_{p,q}X) \Longrightarrow \spi_{p-s,q}X^{hG}\]
with differentials $d_r: E_r^{s,(p,q)} \to E_r^{s+r,(p+r-1,q)}$, which converges completely if $\displaystyle \varprojlim_r{}^1  E_r^{s,(p,q)}=0$ for all $(s,(p,q))$.
\end{prop}

In practice, complete convergence means that the spectral sequence computes $\spi_{*,*}X^{hG}$.  
\begin{rem}\label{rem:descU}
Note that for every fixed $q$ and every smooth scheme $U$, we obtain a spectral sequence of graded abelian groups
\[E_2^{s,(p,q)}(U)=H^s(G, \spi_{p,q}X (U)) \Longrightarrow \spi_{p-s,q}X^{hG}(U).\]
\end{rem}

We will also use a motivic K\"unneth spectral sequence, which was constructed by Dugger and Isaksen in \cite[Proposition 7.7]{dugisa}.
\begin{prop}\label{prop:kun}
Let $E$ be a motivic ring spectrum, $M$ a right $E$--module and $N$ a left $E$--module. If $E$ and $M$ are cellular, then there is a strongly convergent tri-graded spectral sequence
\[ E_{a,(b,c)}^2 = \Tor_{a,(b,c)}^{\pi_{*,*}E}(\pi_{*,*}M, \pi_{*,*}N) \Longrightarrow \pi_{a+b,c}(M \smsh_{E} N )\]
with differentials $d_r : E^r_{a,(b,c)} \to E^r_{a-r, b+r-1,c}$.
\end{prop}

To apply \fullref{prop:kun} to Eilenberg-MacLane spectra, we need the following result of Hoyois \cite[Proposition 8.1]{hoyois}.
\begin{thm}[Hoyois]
The spectrum $\HA$ is cellular for every abelian group $A$.
\end{thm}

Recall that if $R \to T$ is a Galois extension, then $T$ is a projective $R$--module and hence is flat over $R$. Using the following result, this will allow us to relate the homotopy groups of $\HR$ and $\HT$.

\begin{prop}\label{cor:tensor}
Let $R$ be a commutative ring, $M$ be a flat $R$--module and $A$ be any $R$--module. Then 
\[ \underline{\pi}_{p,q} \mathsf{H(A\otimes_R M)} \cong (\spi_{p,q}\HA)  \otimes_R M.\]
\end{prop}

\begin{proof}
This follows from the fact that, for $X$ a Noetherian scheme,
\[H^{p,q}(X, A\otimes_R M) \cong  H^{p,q}(X, A)\otimes_R M.\]
Indeed, let $\mathbb{Z}(q)$ be the motivic complex defined in \cite[Definition 3.1]{motcoh}.
For an abelian group $M$, $M(q) = \mathbb{Z}(q) \otimes M$ and
$H^{p,q}(X, M) = R^p\Gamma(X, M(q))$. However, by the flatness hypothesis, $\Gamma(X, M(q)) \cong \Gamma(X, A(q)) \otimes_R M$. Further, since $ - \otimes_R M$ is an exact functor, $R^p\Gamma(X, M(q)) \cong (R^p\Gamma(X,A(q))) \otimes_R M$, which proves the claim.
\end{proof}

In order to streamline the discussion, we make the following definition.
\begin{defn}\label{defn:neg}
A bigraded abelian group $A_{*,*}$ is \emph{negative} if $A_{p,q} = 0$ whenever
\begin{itemize}
\item $q> 0$, and
\item $q=0$ and $p\neq 0$. 
\end{itemize}
\end{defn}

\noindent
We recognize that this definition seems artificial. However, for any smooth scheme $X$, $\spi_{p,q}\HA(X)$ is negative (see for example \cite[Corollary 4.26]{hoyois}). The next few results depend only on the vanishing properties captured by \fullref{defn:neg}.

Suppose that $M_{*,*}$ and $N_{*,*}$ are $R_{*,*}$--modules, where $R_{*,*}$ is a bi-graded commutative ring that is negative as a bi-graded abelian group. If $M_{*,*}$ and $N_{*,*}$ are negative as bi-graded abelian groups,
then so is any submodule of $M_{*,*}$ and any quotient of $M_{*,*}$. Further, $M_{*,*} \otimes_{R_{*,*}}N_{*,*}$ is negative, and $M_{*,*}$ admits a free resolution by $R_{*,*}$--modules that are also negative.
These observations together imply that $\Tor_{s, (*,*)}^{R_{*,*}}(M_{*,*}, N_{*,*} )$
is negative as a bi--graded abelian group for each $s$. 

\begin{defn}
A motivic spectrum $B$ is \emph{negative} if $\pi_{*,*}B$ is a negative bi-graded abelian group.
\end{defn}

\begin{lem}\label{lem:fixpoints}
If $B$ is a negative spectrum equipped with an action of $G$, then the natural map $B^{hG} \to B$ induces an isomorphism
\[\pi_{0,0}B^{hG} \cong (\pi_{0,0}B)^{G}.\]
\end{lem}

\begin{proof}
For $q=0$ and $U = S$, 
\fullref{rem:descU} gives a spectral sequence
\[E_2^{s,(p,q)} = H^s(G, \pi_{p,0}B) \Longrightarrow \pi_{p-s,0}B^{hG}.\]
The contributions to $\pi_{0,0}B^{hG}$ come from 
 \[E_2^{s,(s,0)} = H^s(G, \pi_{s,0}B) = \begin{cases}  (\pi_{0,0}B)^G & s=0\\
 0 & s \neq 0.
 \end{cases}\]
 Hence, it suffices to prove that elements in $E_r^{0,(0,0)}$ do not support differentials.
The targets of such differentials lie in $ E_r^{r, (r-1,0)}$, which is a sub-quotient of $H^r(G, \pi_{r-1,0}B)$. However,
$\pi_{r-1,0}B = 0$ for $r\geq 2$, since $B$ is negative. Hence, $E_r^{r, (r-1,0)}=0$. 

The map on $\pi_{0,0}$ induced by the natural map $B^{hG} \to B$ factors through the edge homomorphism $\pi_{0,0}B^{hG} \to E_2^{0,(0,0)} \cong (\pi_{0,0}B)^G$ of the spectral sequence, giving the desired isomorphism.
\end{proof}

\begin{prop}\label{prop:homsmash}
Let $A$ and $B$ be negative spectra that are modules over a ring spectrum $R$ that is also negative as a spectrum. If $A$ and $R$ are cellular, or $A$ is $R$--cellular, then the edge homomorphism
\[  \pi_{0,0}A \otimes_{\pi_{0,0}R} \pi_{0,0}B \xra{e} \pi_{0,0}(A \smsh_{R} B) \]
of the K\"unneth spectral sequence is an isomorphism.
\end{prop}
\begin{proof}
In the spectral sequence
\[E^2_{s,(p,q)} = \Tor^{\pi_{*,*}R}_{s,(p,q)}(\pi_{*,*}A, \pi_{*,*}B) \Longrightarrow \pi_{p+s,q} (A \smsh_R B), \]
the contributions to $\pi_{0,0} (A \smsh_R B)$ come from 
\begin{align*}E^2_{0,(0,0)}= \Tor^{\pi_{*,*}R}_{0,(0,0)}(\pi_{*,*}A, \pi_{*,*}B) &\cong (\pi_{*,*}A \otimes_{\pi_{*,*}R} \pi_{*,*}B)_{(0,0)}\\
& \cong \pi_{0,0}A \otimes_{\pi_{0,0}R} \pi_{0,0}B. \end{align*}
The targets of the differentials $d_r:E^r_{0,(0,0)} \to E^r_{-r, (r-1,0)}$ are zero for degree reasons. The sources of the differentials
$d_r:  E^r_{r, (1-r,0)}  \to E^r_{0,(0,0)}$ are zero
since $E^r_{r, (*,*)}$ is negative, and $r\geq 2$ implies that $1-r\neq 0$.
\end{proof}

\begin{prop}\label{prop:galpi0}
Let $A \to B$ be a {homotopical} $G$--Galois extension of cellular, negative, commutative ring spectra. Then $\pi_{0,0}A \to  \pi_{0,0}B$ is a $G$--Galois extension of commutative rings.
\end{prop}
\begin{proof}
The map $\pi_{0,0}A \to (\pi_{0,0}B)^G$ is an isomorphism by \fullref{lem:fixpoints}. Since (\ref{eqn:mapFGB}) is an equivalence, it induces an isomorphism on $\pi_{*,*}$,
which we can precompose with the edge homomorphism to get a map
\begin{align}\label{eqn:map2}
 \pi_{*,*}B \otimes_{\pi_{*,*}A}  \pi_{*,*}B \to \prod_G \pi_{*,*}B.
 \end{align}
Its factors are the composites defined by the following commutative diagram:
\[\xymatrix{ \pi_{*,*}B \otimes_{\pi_{*,*}A}  \pi_{*,*}B  \ar[d]_-{e} \ar[r]^{1 \otimes g} \ar[r]  &  \pi_{*,*}B \otimes_{\pi_{*,*}A}  \pi_{*,*}B \ar[d]_-{e} \ar[dr]^-{m} & \\   
\pi_{*,*}(B\smsh_A B) \ar[r]_-{\pi_{*,*}(1\smsh g)}   & \pi_{*,*}(B\smsh_A B) \ar[r]_-{\pi_{*,*}\mu} &  \pi_{*,*} B. }\]
If we restrict (\ref{eqn:map2}) to $\pi_{0,0}$, the edge homomorphism is an isomorphism, so that the map $\pi_{0,0}B \otimes_{\pi_{0,0}A}  \pi_{0,0}B \to \prod_G \pi_{0,0}B$ is an isomorphism. One verifies easily that this is the map that sends $b_1 \otimes b_2$ to the function $\{g \mapsto b_1 g(b_2)\}$, so that $\pi_{0,0}(A) \to  \pi_{0,0}B$ is indeed a Galois extension.
\end{proof}

{Together with the fact that $\pi_{0,0}\HA =A$,} one direction of \fullref{thm:galEM} follows immediately from \fullref{prop:galpi0}, and we record it here.
\begin{prop}\label{prop:EMforward}
Let $R \to T$ be a homomorphism of commutative rings. If $\HR \to \HT$ is a {homotopical} $G$--Galois extension of motivic ring spectra, then $R \to T$ is a $G$--Galois extension of commutative rings.
\end{prop}

Before proving the converse, we note the following result.
\begin{lem}\label{lem:MN}
Let $R \to A$ and $R \to B$ be maps of commutative rings.  If $B$ is projective as an $R$--module, then the natural map
\[\HA \smsh_{\HR} \HB \to \mathsf{H(A \otimes_R B)}\] 
is an {isomorphism in $\SH{S}$.}
\end{lem}

\begin{proof}
Since $B$ is projective as an $R$--module, it follows from \fullref{cor:tensor} that $\pi_{*,*}\HB \cong \pi_{*,*}\HR \otimes_R B$ and that $\pi_{*,*} \mathsf{H(A \otimes_R B)} \cong  \pi_{*,*}\HA  \otimes_{R} B $. Further, the projectivity of $B$ implies that $\pi_{*,*}\HR \otimes_R B$ is a projective $\pi_{*,*}\HR$--module. Combining this with the fact that $\HA$ is cellular, \cite[Lemma 5.2]{hoyois} implies that the edge homomorphism
\[\pi_{*,*}\HA \otimes_{\pi_{*,*}\HR} \pi_{*,*}\HB  \to  \pi_{*,*}(\HA \smsh_{\HR} \HB)  \]
is an isomorphism. We have a commutative diagram, where all maps are obvious universal maps,
\[\xymatrix{ \pi_{*,*}\HA \otimes_R B \ar[r]^-{\cong}  \ar[d]_-{\cong} & \pi_{*,*}\HA \otimes_{\pi_{*,*}\HR} \pi_{*,*}\HB \ar[d]^-{\cong} \\
\pi_{*,*} \mathsf{H(A \otimes_R B)} & \pi_{*,*}(\HA \smsh_{\HR} \HB).  \ar[l]   }
 \]
It follows that the bottom map also induces an isomorphism on $\pi_{*,*}$. Since $\HA \smsh_{\HR} \HB$ is a co-equalizer of cellular objects, it is also cellular. Equivalences of cellular motivic spectra are detected by $\pi_{*,*}$, so this proves the claim.
\end{proof}

\begin{lem}\label{lem:EMrev1}
If $R \to T$ is a $G$--Galois extension of commutative rings, then the map
\begin{align*}
\HT \smsh_{\HR} \HT  \to F(G_{+}, \HT),
\end{align*}
which is the transpose of $\HT \smsh_{\HR}\HT \smsh G_{+}  \to \HT$, is {an isomorphism in $\SH{S}$}.
\end{lem}
\begin{proof}
If $R \to T$ is a Galois extension, then $T$ is a finitely generated projective $R$--module.  In particular, it is flat over $R$. By \fullref{lem:MN}, the canonical map $\HT \smsh_{\HR} \HT \to \mathsf{H(T \otimes_R T)}$ 
is {an isomorphism in $\SH{S}$}. By assumption, the map
\[T \otimes_R T  \to \prod_{G} T, \]
of which the factors are the composites $T \otimes_R T \xra{1 \otimes g} T \otimes_R T \xra{m} T$, 
is an isomorphism of commutative rings. So the induced map
\[ \mathsf{H(T \otimes_R T)} \to \mathsf{H\left(\prod_{G} T \right)}  \]
is {an isomorphism in $\SH{S}$}. Further, $\mathsf{H(\prod_{G} T)} \simeq \prod_G \HT$, and the factors of
the map $\mathsf{H(T \otimes_R T)} \to \prod_{G} \HT$ are the composites 
\[\mathsf{H(T \otimes_R T  )} \xra{\mathsf{H}(1\otimes g)} \mathsf{H(T \otimes_R T)} \xra{\mathsf{H}(m)} \HT.  \]
Since the following diagram commutes
\[\xymatrix{  \HT \smsh_{\HR} \HT \ar[r]^{1 \smsh g} \ar[d] & \HT \smsh_{\HR} \HT \ar[dr]^-{\mu} \ar[d] &  \\
\mathsf{H(T \otimes_R T  )} \ar[r]_-{\mathsf{H}(1\otimes g)}  & \mathsf{H(T \otimes_R T)} \ar[r]_-{\mathsf{H}(m)} & \HT, }\]
the map $\HT \smsh_{\HR} \HT \to  \prod_G \HT $ with factors
\[\HT \smsh_{\HR} \HT \xra{1 \smsh g} \HT \smsh_{\HR} \HT \xra{\mu} \HT\]
is {also an isomorphism  in $\SH{S}$}.
\end{proof}

The last condition needed to prove \fullref{thm:galEM} is checked in the following result.

\begin{prop}\label{prop:EMreverse}
If $R \to T$ is a $G$--Galois extension of commutative rings, then $\HR \to \HT^{hG} $ is {an isomorphism in $\SH{S}$}. 
\end{prop}

\begin{proof}
Fix $U \in \Sm{S}$. By \fullref{rem:descU}, there is a  homotopy fixed point spectral sequence 
\[E_2^{s,(p,q)}(U)  = H^s(G, \spi_{p,q}\HT(U)) \Longrightarrow \spi_{p-s,q}\HT^{hG}(U).\]
Further, since $T$ is a projective $R$--module, it is flat. By \fullref{cor:tensor}, we have
\begin{align*}
E_2^{s,(p,q)}(U) 
\cong \Ext^s_{R[G]}(R, \spi_{p,q}\HT(U) )  \cong \Ext^s_{R[G]}(R,  \spi_{p,q}\HR(U)  \otimes_R T),\end{align*}

Let $T^\vee =\Hom_R(T, R)$ with left $G$ action given by $(g\phi)(t) = \phi(g^{-1}t)$. 
Since $T$ is finitely generated 
as an $R$--module, $(T^\vee)^\vee\cong T$ as $R[G]$--modules, and
\[ \spi_{p,q}\HR(U)  \otimes_R T\cong  \Hom_R(T^\vee,  \spi_{p,q} \HR(U)). \]
Now, note that $T$ is also projective and finitely generated as an $R[G]$-module \cite[Proposition 2.3.4(c)]{rognes}). That is, there exists an $R[G]$--module $Q$ such that $Q \oplus T \cong R[G]^n$. It follows that
\[\Hom_R(T, R) \oplus \Hom_R(Q, R) \cong \Hom_R(R[G]^n, R) \cong R[G]^n   \]
as $R[G]$--modules, so that $T^{\vee}$ is a projective $R[G]$--module. 
Therefore, given a resolution $P^\ast\to R$ by projective $R[G]$-modules, 
$P^\ast \otimes_{R[G]} T^\vee \to R\otimes_{R[G]} T^\vee$ is a resolution of $R\otimes_{R[G]} T^\vee$ by projective $R$--modules. 
It then follows that
\[\Ext^s_{R[G]}(R,   \Hom_R(T^\vee,  \spi_{p,q} \HR(U))) \cong \Ext^s_{R}(T^\vee  \otimes_{R[G]} R  ,\spi_{p,q}\HR(U) ).\]
Finally, since $T^\vee  \otimes_{R[G]} R$ is finitely generated as an $R$--modules, $T^\vee  \otimes_{R[G]} R \cong ((T^\vee  \otimes_{R[G]} R)^{\vee})^{\vee}$.
However,
\begin{align*} 
 (T^\vee  \otimes_{R[G]} R)^{\vee} &\cong \Hom_{R[G] } (R, \Hom_R(T^\vee, R)) \\
 &\cong  \Hom_{R[G]}(R, T)\cong T^G\cong R,
\end{align*} 
so that $T^\vee  \otimes_{R[G]} R \cong R$. 

We conclude that
\begin{align*} 
E_2^{s,(p,q)}(U)
& \cong \Ext^{s}_R(R, \spi_{p,q}\HR(U)) \cong \begin{cases}  \spi_{p,q}\HR(U) & s =0 \\
0 &\text{otherwise.}
\end{cases} \end{align*}
The spectral sequence collapses and converges completely. In particular, the edge homorphism $\pi_{*,*}\HR(U) \to (\pi_{*,*}\HT(U))^G \cong \pi_{*,*}\HT^{hG}(U)$ is an isomorphism.
Hence, the map $\HR\to \HT$ induces an isomorphism on $\spi_{*,*}$.
\end{proof}

 \fullref{thm:galEM} follows immediately from \fullref{prop:EMforward}, \fullref{lem:EMrev1}, and \fullref{prop:EMreverse}.

\subsection{The extension $\KQ \to \KGL$}\label{sec:kgl}
In \cite[Proposition 5.3.1]{rognes}, Rognes proved that the complexification $KO \ra KU$
is a faithful Galois extension of ring spectra. In this section, we consider the analogous map $\KQ \to \KGL$ of motivic spectra. Here, $\KQ$ is the Hermitian $K$--theory spectrum (see Hornbostel \cite{hornbostel-a1}) and $\KGL$ is the algebraic $K$--theory spectrum (see Voevodsky et al. \cite[Section 3.2]{mot-hom-book-nordfjordeid}). The cyclic group $C_2$ acts on $\KGL $ by  sending a vector bundle to its dual.

The problem of whether or not $\KQ \to \KGL^{hC_2}$ is a stable weak equivalence is closely related to the homotopy limit problem of Thomason \cite{thomason}. This problem was rephrased by Willams \cite[p.627]{kthyhb} as the question of whether or not $\KQ \to \KGL^{hC_2}$ becomes an equivalence after profinite completion. Hu, Kriz and Ormsby in \cite[Theorem 20]{hko} proved that if the base scheme $S$ is $Spec(k)$ for $k$ such that the $2$--cohomological dimension of $k[i]$ is finite, then $\KQ \to \KGL^{hC_2}$ is an equivalence after $2$-completion. More recently, Berrick et al. in \cite[Section 2]{bako} gave a very detailed account of the solution to William's conjecture. In particular, in \cite[Theorem 2.4, Corollary 2.6]{bako}, they give precise conditions on $S$ under which the map $\KQ \to \KGL^{hC_2}$ is an equivalence, even before $2$--completion. This result is the key to the proof of the following theorem.

\begin{thm}\label{thm:simKQ} Let $S$ be a scheme on which $2$ is invertible and which satisfies the following additional conditions:
\begin{enumerate}
\item $S $ is noetherian of finite Krull dimension,
\item $S$ has finite $2$-primary virtual cohomological dimension,
\item $S$ has an ample family of line bundles, and
\item $-1$ is a sum of squares in all the residue fields of $S$.
\end{enumerate}
Then, over $S$,  the natural map $\KQ \to \KGL$ is a $C_2$--Galois extension of ring spectra, which is faithful on $\eta$--complete $\KQ$--modules. 
\end{thm} 

Together with \cite[Theorem 2.4, Corollary 2.6]{bako}, the results of \fullref{prop:mapKGL} and \fullref{prop:faithKGL} below prove \fullref{thm:simKQ}.

\begin{prop}\label{prop:mapKGL}
The natural map $g: \KGL \wedge_{\KQ} \KGL \to F((C_2)_+, \KGL)$ is {an isomorphism in $\SH{S}$}.
\end{prop}

\begin{proof}
The argument is the motivic analogue of the proof of Rognes \cite[Proposition 5.3.1]{rognes}, where he shows that $KU \wedge_{KO} KU \to  F((C_2)_+, KU)$ is a weak equivalence. In the proof, we use the notation of R\"ondigs-{{\O}}stv\ae{}r \cite{ro} to denote the maps. 
We warn the reader that the names of some of our maps are not the same as their analogues in Rognes \cite{rognes}. 
Let $\beta: \mmS^{2,1} \smsh \KGL \to \KGL$ be the Bott map. Consider the cofiber sequence \cite[(5)]{ro},
\begin{equation}\label{eqn:etaseq}
\mmS^{1,1} \smsh \KQ \xra{\eta} \KQ \xra{f} \KGL \xra{\partial} \mmS^{2,1}\smsh \KQ  .\end{equation}
By \cite[Theorem 4.4]{ro}, $\partial$ factors as
\[\KGL \xra{\beta^{-1}} \mmS^{2,1}\smsh \KGL \xra{\mmS^{2,1}\smsh h} \mmS^{2,1}\smsh \KQ\]
where $h$ is the hyperbolic map, the motivic analogue of realification.
We apply $\KGL\smsh_{\KQ}(-)$ to (\ref{eqn:etaseq}) to construct a diagram
\[\xymatrix{ \KGL\smsh_{\KQ} \KQ \ar[r] \ar[d] & \KGL\smsh_{\KQ} \KGL \ar[r]^-{1\smsh \partial} \ar[d]^-{g}&  \KGL\smsh_{\KQ} (\mmS^{2,1}\smsh \KQ) \ar[d]^{1 \smsh \beta} \\
\KGL  \ar[r]^-{\Delta} & F((C_2)_+, \KGL) \ar[r]^-{\delta} &   \KGL }\]
The map $\Delta$ is the map induced by $(C_2)_+ \to *_+$. This is the motivic analogue of the trivial Galois extension in the sense of Rognes \cite[Section 5.1]{rognes}. As in Rognes, the map $\delta$ is the difference of the two projection maps $F((C_2)_+, \KGL)  \simeq \prod_{C_2} \KGL \to \KGL$.

The left hand square commutes strictly since $\KQ$ has a trivial $C_2$--action, so we focus on the right hand square. We precompose it with the Bott map and restrict along $\KQ \to \KGL$ to obtain a diagram
\[\xymatrix{   \KQ \smsh_{\KQ}(\mmS^{2,1}\smsh \KGL) \ar[rr]^{f \smsh \mmS^{2,1}{\smsh} h} \ar[d]^-{g \circ (f \smsh \beta)}& &  \KGL\smsh_{\KQ} (\mmS^{2,1}\smsh \KQ) \ar[d]^{1 \smsh \beta} \\
 F((C_2)_+, \KGL) \ar[rr]^-{\delta} & & \KGL. }\]
 It suffices to prove that the two compositions are equivalent as $\KQ$--module maps. First, note that there is a commutative diagram
\[\xymatrix{  \KQ \smsh_{\KQ}(\mmS^{2,1}\smsh \KGL)  \ar[rr]^-{c \smsh \mmS^{2,1}{\smsh} h} \ar[d] & &  \KGL\smsh_{\KQ} (\mmS^{2,1}\smsh \KQ) \ar[d]  \\
 \mmS^{2,1}\smsh \KGL \ar[rr]^-{(\mmS^{2,1}{\smsh} f)\circ (\mmS^{2,1}{\smsh} h)}  & &  \mmS^{2,1}\smsh \KGL  }.\]
As in Rognes \cite[Proposition 5.3.1]{rognes}, we must prove that
\[   \beta  \circ  (\mmS^{2,1}{\smsh} f )\circ (\mmS^{2,1}{\smsh} h)  \simeq \delta \circ g \circ (f \smsh \beta).\]
The key identities are $\delta \circ g \simeq \mu (1 \smsh (1 - t))$ where $t $ is the action of the generator of $(C_2)_+$, denoted by $\Psi_{st}^{-1}$ in R\"ondigs-{{\O}}stv\ae{}r \cite[Section 5.2]{ro}, $f\circ h \simeq 1+t$ and $\beta \circ \mmS^{2,1}(1+t) \simeq (1-t) \circ \beta$. The first is formal and thus holds in our setting. The second is R\"ondigs-{{\O}}stv\ae{}r \cite[Equation (17)]{ro} and third follows from \cite[Equation (16)]{ro}.
\end{proof}

\begin{prop}\label{prop:faithKGL}
The map $\KQ \to \KGL$ is faithful on the subcategory of $\eta$--complete $\KQ$--modules.
\end{prop}
\begin{proof}
Suppose $N$ is an  $\eta$-complete $\KQ$-module such that $N\wedge_{\KQ}\KGL \cong \ast$ in $\SH S$. We want to conclude that $N$ itself is contractible. Smashing the cofiber sequence 
\[ \mmS^{1,1}\smsh \KQ \xrightarrow{\eta} \KQ \to \KGL \]
with $N$, we see that $\eta \smsh 1_N$ is {an isomorphism in $\SH{S}$}. The $\eta$-completion is the homotopy limit of $C(\eta^k) \smsh N$, where $C(\eta^k)$ is the cofiber of $\eta^k$. However, since $\eta^k\smsh 1_N$ is an {isomorphism}, $C(\eta^k) \smsh N \cong \ast$. Hence, $N \simeq N^{\wedge}_{\eta} \cong \ast$.
\end{proof}

\begin{rem}
The Galois extension $\KQ \to \KGL$ is not faithful on all $\KQ$--module motivic spectra. For instance, let $\KT$ be the Balmer-Witt theory spectrum, obtained from $\KQ$ by inverting $\eta$. Let $\KQ \to \KT$ be the natural map. Then 
smashing the sequence
\[ \mmS^{1,1} \smsh \KQ \xrightarrow{\eta} \KQ \to \KGL \]
with $\KT$, we conclude that $\KT \wedge_{\KQ } \KGL \cong \ast$ in $\SH S$. On the other hand, $\KT$ itself is non-trivial. In particular, it has interesting $2$-torsion homotopy (see R\"ondigs-{{\O}}stv\ae{}r \cite{ro} for its slices).

In general, if $N\wedge_{\KQ}\KGL$ is contractible, then the $\eta$-completion of $N$ is contractible.
\end{rem}

\section{Future directions}\label{sec:futuredirections}

\subsection{\'Etale Galois Extensions}\label{sec:etalegalext}

In this paper we have focused on \emph{simplicial Galois extensions}, i.e., extensions that are Galois with respect to homotopy fixed points that are defined in terms of the geometric realization $EG$ of an appropriate replacement of the simplicial $G$-set $E_{\bullet}G$, viewed as a constant simplicial presheaf in $\Gmot$ (see \fullref{lem:unstable_fibrant_repl}). 
The motivic $G$-space $EG$ is a free contractible $G$-space, but is not universal with this property, as the following result, which we learned from Gepner and Heller, illustrates.

\begin{lem}
There exists a motivic $G$-space with free action for which there is no $G$-equivariant map to $EG$.
\end{lem}

\begin{proof}
Let  $S=\Spec(k)$ for a field $k$, and suppose $L/k$ is a finite Galois extension such that $G \subset \Gal(L/k)$. It follows that $\Spec L$ is a free $G$-scheme via the Galois action. 
Indeed, to show that $(\Spec L)^H$ is empty for every non-trivial subgroup $H$ of $G$, it suffices to show that $\Spec L$ admits no maps from $\Spec F $, where $F$ is any field. By adjunction,
\[ \Hom_{\Sm{S}} (\Spec F, (\Spec L)^H) \cong \Hom_{H\text{-}\Sm{S}}(\Spec F, \Spec L) \cong  \Hom_{H}(L, F).\]
However, since any map of fields is an inclusion, there are no $H$-equivariant maps $L\to F$ because $F$ has trivial action and $L$ does not.

To establish that there are no $G$-equivariant maps from $\Spec L$ to $EG$, we prove that the set $[\Spec L, EG]_G $ of $G$-homotopy classes of $G$-maps
 is empty.  Since the Bousfield-Kan map $EG \to |E_{\bullet}G|$ is an equivariant weak equivalence, this is equivalent to checking that $\big[\Spec L, |E_{\bullet}G|\big]_{G}$ is empty. By Morel and Voevodsky \cite[Corollary 2.3.22]{morelvo}, the map $ EG^{(0)}(\Spec L) \to \big[\Spec L, |E_{\bullet}G|\big]_G$ is surjective, where $EG^{(0)}$ is the zero skeleton of $|E_{\bullet}G|$, which is simply $\coprod_{G} \Spec k$.  It suffices therefore to prove that   $ EG^{(0)}(\Spec L)$ is empty.  
 
 Because $\Spec L$ is connected,  any map $f : \Spec L \to EG^{(0)}$  must factor through one of the components $\Spec k$ of the target.  Since the inclusion of a component into $EG^{(0)}$ is not equivariant, the map $f$ is not equivariant either.
\end{proof}

There is another free contractible $G$-space that does have the desired universal property. The \emph{\'etale} or \emph{geometric} $\EGgeom$ can be constructed as 
$$\EGgeom= \colim_n  \left( \bA(n\rho_G)-\bigcup_{e\neq H\subset G}\bA(n\rho_G)^H\right),$$  where $\rho_G$ is the regular representation of $G$, and the maps in the colimit are induced by the inclusions $n\rho_G \to (n+1)\rho_G$. 

If we set up appropriate model structures, then for nice enough motivic $G$-spaces $X$, we could define the \emph{\'etale} homotopy fixed points of $X$ to be $X^{h_{\text{\'et}}G} = \Hom(\EGg{G}_+, X)^{G}$. The corresponding theory of \emph{\'etale Galois extensions} should provide rich examples that we cannot capture with the simplicial version. For example, the following result, which the authors learned from Jeremiah Heller, is an analogue of Willams's conjecture for \'etale homotopy fixed points.

\begin{thm}\label{prop:hell}
The natural map $\KQ \to \KGL^{h_{\text{\'et}}C_2}$ is an equivalence of motivic spectra, as long as $2$ is invertible in the base $S$.
\end{thm}
\noindent
Note that \fullref{prop:hell} holds before $2$--completion, and over a general base in which $2$ is invertible, i.e., the additional conditions of \fullref{thm:simKQ} are not required. Given the right framework, this would imply that $\KQ\to \KGL$ is an \'etale Galois extension.

One of the advantages of defining homotopy fixed points using $EG$ is that we can use the associated homotopy fixed point spectral sequence of \fullref{prop:hfss}, which is well understood. While $\EGg{G}$ can be given a skeletal filtration, the authors do not have a good understanding of the associated spectral sequence.

\subsection{Motivic Hopf-Galois extensions}\label{sec:hopfgalext}

Extending the notion of a Galois extension, Rognes \cite[Chapter 12]{rognes}, Hess \cite{hess-hg} and Hess--Berglund \cite{hessberg-hg} have developed a theory of homotopic Hopf-Galois extensions for a Hopf--algebra $H$. For example, if $H$ is the $\Bbbk$-dual of a group ring  $\Bbbk [G]$, one recovers the usual notion of Galois extension of $\Bbbk$-algebras..

For classical ring spectra, Rognes defined Hopf-Galois extensions as follows. Let $\phi: A \to B$ be a morphism of commutative ring spectra, and let $H$ be a ring spectrum equipped with a comultiplication $H \to H \smsh H$ that is a map of ring spectra, i.e., $H$ is a bialgebra spectrum. Suppose that $H$ coacts on $B$ so that $\phi$ is a morphism of $H$-comodules when $A$ is endowed with the trivial $H$-coaction. 
Consider the following maps.
\begin{enumerate}[(a)]
\item The Galois map $\beta_{\phi} : B \smsh_A B \to  B \smsh H$, which is the composite of the co-action on the right factor of $B$ followed by the multiplication of the left factors.
\item The natural map from $A$ to the homotopy coinvariants of the $H$--coaction on $B$, $A \to B^{hco H}$. Here, $ B^{hco H}$ is defined as $\Tot C^{\bullet}(H;B)$ for $C^{\bullet}(H;B)$ the Hopf cobar complex for $H$--coacting on $B$.
\end{enumerate}
If (a) and (b) are both weak equivalences, then $\phi: A \to B$ is a \emph{homotopic $H$-Hopf-Galois extension} in the sense of Rognes.

In  \cite[Section 12.2]{rognes}, Rognes proves that the map $\mathbb{S} \to MU$ is an $\mathbb{S}[BU]$-Hopf-Galois extension of spectra. The key input is the Thom isomorphism, which implies that 
\begin{align}\label{eqn:hMU}
h: MU \smsh MU \to MU \smsh \mathbb{S}[BU]
\end{align}
is a weak equivalence. This, in turn, implies that $\mathbb{S}^{\smsh}_{MU} \simeq  \Tot(C^{\bullet}(\mathbb{S}[BU];MU))$ (see \cite[Proposition 12.1.8]{rognes}). By the convergence of the Adams-Novikov spectral sequence, $\mathbb{S} \to \mathbb{S}^{\smsh}_{MU}$ is a weak equivalence.

With an adequate model category structure on $H$-comodule algebras in motivic spectra, one should be able to adapt this theory to study Hopf-Galois extensions of motivic spectra. In particular, one should then have the following example. 

Let $\MGL$ be the algebraic cobordism spectrum (see Voevodsky et al. \cite[Section 3.3]{mot-hom-book-nordfjordeid}).  Motivically, there is a Thom isomorphism (see, for example, Dugger and Isaksen \cite[Remark 8.10]{dugisa}). Hence,
\[\MGL \smsh \MGL \to \MGL \smsh \mmS[\BGL]\]
is an equivalence. As above, this implies that
\[ \mmS^{\smsh}_{\MGL} \to \Tot C^{\bullet}(\mmS[\BGL]; \MGL)\]
is a weak equivalence. However, since the motivic Adams-Novikov spectral sequence converges to the homotopy groups of $\mmS^{\wedge}_{\eta}$,
we will be able to conclude that
the map $\mmS^{\wedge}_{\eta} \to \MGL$ is an $\mmS[\BGL]$-Hopf-Galois extension, once we have set up the framework properly.

\appendix
\section{Model structure techniques}\label{appendix}

In this appendix we recall techniques from {Hess et al.}\cite{HKRS} for establishing the existence of induced model category structures.

\begin{notation}\label{notation:cof} For any class of maps $X$ in a category $\sM$, we let $\LLP(X)$ (respectively, $\RLP(X)$) denote the class of maps having the left lifting property (respectively, the right lifting property) with respect to all maps in $X$. We use notation $X$-cof for the class of maps $\LLP(\RLP(X))$.\end{notation}

\begin{defn} A \emph{weak factorization system} on a category $\sfC$ consists of a pair $(\mathcal L,\mathcal R)$ of classes of maps so that the following conditions hold.
\begin{itemize}
\item Any morphism in $\sfC$ can be factored as a morphism in $\mathcal L$ followed by a morphism in $\mathcal R$.
\item $\mathcal L = \LLP(\mathcal R)$ and $\mathcal R = \RLP(\mathcal L)$.
\end{itemize}
\end{defn}

In particular, if $(\sM,\Fib,\Cof,\WE)$ is a model category, then $(\Cof \cap \WE, \Fib)$ and $(\Cof, \Fib \cap \WE)$ are weak factorization systems.  If one additional condition is satisfied, the converse holds as well.

\begin{prop}[{Joyal and Tierney} {\cite[7.8]{joyal-tierney}}]\label{prop:model-via-wfs} If $\sM$ is a bicomplete category, and $\Fib, \Cof, \WE$ are classes of morphisms so that 
\begin{itemize}
\item $\WE$ satisfies the 2-of-3 property, and 
\item $(\Cof \cap \WE, \Fib)$ and $(\Cof, \Fib \cap \WE)$ are weak factorization systems,
\end{itemize}
then $(\sM,\Fib,\Cof,\WE)$ is a model category.
\end{prop}

\begin{defn}
Let $(\sM,\Fib,\Cof,\WE)$ be a model category and consider a pair of adjunctions
\[ \xymatrix@C=4pc@R=4pc{ \sK \ar@<1ex>[r]^V \ar@{}[r]|\perp & \sM \ar@<1ex>[l]^R \ar@<1ex>[r]^L\ar@{}[r]|\perp & \sfC \ar@<1ex>[l]^U}\]
where the categories $\sfC$ and $\sK$ are bicomplete.
If they exist:
\begin{itemize}
\item the \emph{right-induced model structure} on $\sfC$ is given by \[\Big(\sfC,U^{-1}\Fib, \LLP\big({U^{-1}(\Fib \cap\WE)}\big), U^{-1}\WE\Big),\] and
\item the \emph{left-induced model structure} on $\sK$ is given by \[ \Big(\sK, \RLP{\big(V^{-1}(\Cof\cap\WE)\big)}, V^{-1}\Cof, V^{-1}\WE\Big).\]
\end{itemize}
\end{defn}

\begin{rem} The adjunction $(L,U)$ is a Quillen pair with respect to the right-induced model category structure on $\sfC$, when it exists.  Dually, the adjunction $(V,R)$ is a Quillen pair with respect to the left-induce model category structure on $\sK$, when it exists.
\end{rem}

Establishing an induced model category structure therefore reduces to proving the existence of appropriate weak factorization systems and checking a certain acyclicity condition.

\begin{prop}\label{prop:acyclicity-reduction}\cite[Proposition 2.14]{HKRS} Suppose $(\sM, \Fib, \Cof,\WE)$ is a model category, $\sfC$ and $\sK$ are bicomplete categories, and there exist adjunctions 
\[ \xymatrix@C=4pc@R=4pc{ \sK \ar@<1ex>[r]^V \ar@{}[r]|\perp & \sM \ar@<1ex>[l]^R \ar@<1ex>[r]^L\ar@{}[r]|\perp & \sfC \ar@<1ex>[l]^U}\] so that the right-induced weak factorization systems exists on $\sfC$, and the left-induced weak factorization systems exists on $\sK$.  It follows that
\begin{enumerate}
\item the right-induced model structure exists on $\sfC$ if and only if \[\LLP\,(U^{-1}\Fib) \subset U^{-1}\WE;\] and
\item the left-induced model structure exists on $\sK$ if and only if \[\RLP{(V^{-1}\Cof)} \subset V^{-1}\WE.\]
\end{enumerate}
\end{prop}

Under reasonable conditions on the categories $\sM$, $\sfC$, and $\sK$, the desired right- and left-induced weak factorization systems are guaranteed to exist, so that only the acyclicity conditions remains to be checked.

\begin{cor}\label{cor:cofib-gen}\cite[Corollaries 3.1.7 and 3.3.4]{HKRS} Suppose $(\sM, \Fib, \Cof,\WE)$ is a locally presentable, cofibrantly generated model category, $\sfC$ and $\sK$ are locally presentable categories, and there exist adjunctions 
\[ \xymatrix@C=4pc@R=4pc{ \sK \ar@<1ex>[r]^V \ar@{}[r]|\perp & \sM \ar@<1ex>[l]^R \ar@<1ex>[r]^L\ar@{}[r]|\perp & \sfC. \ar@<1ex>[l]^U}\] 
It follows that
\begin{enumerate}
\item the right-induced model structure exists on $\sfC$ if and only if \[\LLP\,(U^{-1}\Fib) \subset U^{-1}\WE;\] and
\item the left-induced model structure exists on $\sK$ if and only if \[\RLP{(V^{-1}\Cof)} \subset V^{-1}\WE.\]
\end{enumerate}
\end{cor}

The following consequences of \fullref{cor:cofib-gen} are frequently applied in \fullref{sec:motmodstructures} of this paper.

\begin{prop}\label{prop:left_ind_along_leftQuillen} Consider a Quillen adjunction between locally presentable categories
\[ \xymatrix@C=4pc{ \sK \ar@<1ex>[r]^V \ar@{}[r]|\perp & \sM, \ar@<1ex>[l]^R}\]
where $\sM$ is a cofibrantly generated model category and $\sK$ is a model category.
If $V$ preserves all weak equivalences, then the left-induced model structure created by $V$ exists on $\sK$. 
\end{prop}

\begin{proof} By \fullref{cor:cofib-gen} we need only to check the acyclicity condition  
\[\RLP{(V^{-1}\Cof_\sM)} \subseteq V^{-1}\WE_\sM.\]
Since $\Cof_\sK \subseteq V^{-1}\Cof_\sM$, it follows that 
$$\RLP{(V^{-1}\Cof_\sM)} \subseteq \RLP(\Cof_\sK) = \WE_\sK \cap \Fib_\sK \subseteq \WE_\sK.$$ 
By hypothesis  $ \WE_\sK \subseteq V^{-1}\WE_\sM$, which finishes the proof.
\end{proof}

Similarly, a right-induced model structure along a right Quillen functor $U$ exists if $U$ preserves all weak equivalences.

\begin{prop} Let \[ \xymatrix@C=4pc{ \sK \ar@<1ex>[r]^V \ar@{}[r]|\perp & \sM. \ar@<1ex>[l]^R}\] be an adjunction between locally presentable categories.  If  there is a cofibrantly generated model structure on $\sM$, $(\Cof, \Fib, \WE)$, such that the left-induced model structure created by $V$ exists on $\sK$, then for any cofibrantly generated model structure on $\sM$, $(\widetilde{\Cof}, \widetilde{\Fib}, \widetilde{\WE})$, such that $\Cof \subseteq \widetilde{\Cof}$ and $\WE \subseteq \widetilde{\WE} $, there exists a left-induced model structure on $\sK$ created by $V$ from $(\widetilde{\Cof}, \widetilde{\Fib}, \widetilde{\WE})$.
\end{prop}

\begin{proof}It is enough to check the acyclicity condition $RLP(V^{-1}(\widetilde{\Cof})) \subseteq V^{-1}\widetilde{\WE}$, which follows from the sequence of inclusions
$$RLP(V^{-1}\widetilde{\Cof}) \subseteq RLP(V^{-1}\Cof) \subseteq V^{-1}\WE \subseteq V^{-1} \widetilde{\WE}.$$  
\end{proof}

\begin{cor}\label{cor:left_ind_left_Bousf} Let \[ \xymatrix@C=4pc{ \sK \ar@<1ex>[r]^V \ar@{}[r]|\perp & \sM. \ar@<1ex>[l]^R}\] be an adjunction between locally presentable categories.  If there is a model structure on $\sM$, $(\Cof, \Fib, \WE)$, such that the left-induced model structure created by $V$ exists on $\sK$, then there exists a left-induced model structure created by $V$ on $\sK$ from any left Bousfield localization of $(\Cof, \Fib, \WE)$.
\end{cor}

We recall \cite[Theorem 2.3.2]{HKRS} below, which is crucial in \fullref{section:comalg_motSp}.

\begin{thm}\label{thm:square}\cite[Theorem 2.3.2]{HKRS} Given a square of adjunctions 
\[ \xymatrix@R=4pc@C=4pc{ \sK \ar@{}[r]|-{\perp} \ar@<-1ex>[d]_-L \ar@<-1ex>[r]_-{R} & \sM  \ar@<1ex>[d]^-L \ar@<-1ex>[l]_-V \\  
\ar@{}[r]|-{\top} \ar@{}[u]|-{\dashv} \sN \ar@<1ex>[r]^-{R} \ar@<-1ex>[u]_-U & \sP \ar@<1ex>[u]^-U \ar@<1ex>[l]^-V\ar@{}[u]|-{\vdash} }\] 
between locally presentable categories, suppose that $(\sK,\Cof_\sK,\Fib_\sK,\WE_\sK)$ is a model category such that the left-induced model structure  created by $V$, denoted $(\sM,\Cof_\sM,\Fib_\sM,\WE_\sM)$, and the right-induced model structure created by $U$, denoted $(\sN,\Cof_\sN,\Fib_\sN,\WE_\sN)$ both exist.

If $UV\cong VU$, $LV\cong VL$ (or, equivalently, $UR\cong RU$), 
then there exists a right-induced model structure on $\sP$, created by $U\colon \sP \lra \sM$, and a left-induced model structure on $\sP$, created by $V\colon \sP \lra \sN$, so that the identity is a left Quillen functor from the right-induced model structure to the left-induced one:
\[ \xymatrix@C=4pc@R=4pc{  \sP_{\mathrm{right}} \ar@<1ex>[r]^-\id \ar@{}[r]|-\perp & \sP_{\mathrm{left}}. \ar@<1ex>[l]^-\id}\]
\end{thm}

\begin{rem} All the results above can be generalized from cofibrantly generated model categories to \emph{accessible} model categories in the sense of \cite{HKRS}. For the definition and general properties of an accessible model category, see \cite{HKRS}.
\end{rem}

\newpage
\section{Glossary of model structures}\label{glossary}
\begin{center}
\includegraphics[height=0.95\textheight]{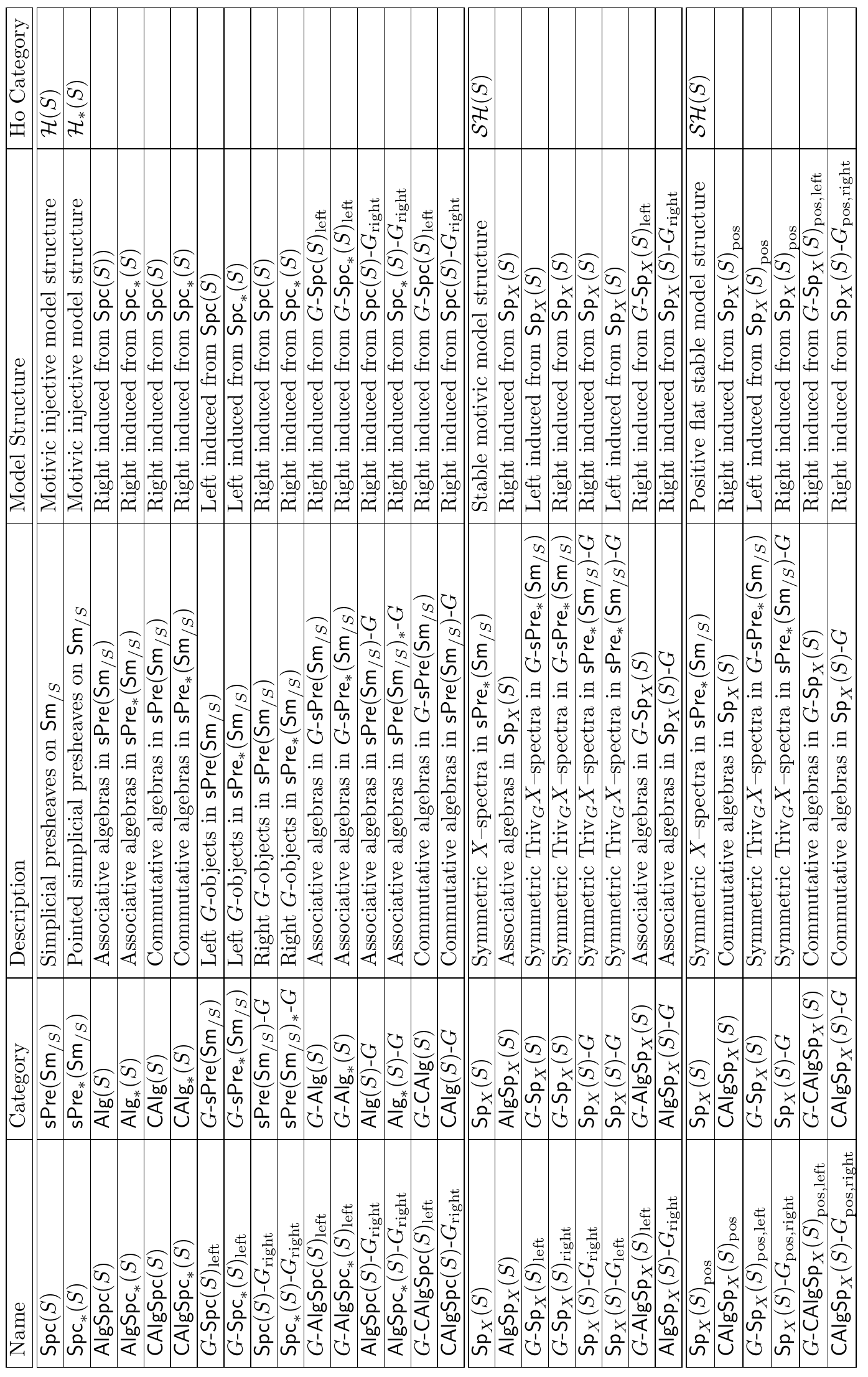}
\end{center}

\bibliographystyle{plain}

\end{document}